\documentclass[11pt,letterpaper]{amsart}
\usepackage{amsmath, amsthm}
\usepackage{amsfonts}
\usepackage{amssymb}
\usepackage{enumerate}
\usepackage{graphicx}
\usepackage[table]{xcolor}
\usepackage[all,arc,curve,color,frame]{xy} % til at tegne grafer og diagrammer

\usepackage{amsmath}
\usepackage{amsfonts}
\usepackage{amssymb}
\usepackage{graphicx}%
\usepackage{color}
\usepackage{amsmath}
\usepackage{amsfonts}
\usepackage{amssymb}
\usepackage{graphicx}%
\usepackage{xypic}
\usepackage{tikz-cd}
\usepackage{color}
\usepackage{mathbbol}
\providecommand{\U}[1]{\protect\rule{.1in}{.1in}}
\newtheorem{theorem}{Theorem}

\newtheorem{corollary}[theorem]{Corollary}

\newtheorem{lemma}[theorem]{Lemma}

\newtheorem{proposition}[theorem]{Proposition}
\newtheorem{remark}[theorem]{Remark}

\tikzset{
  symbol/.style={
    draw=none,
    every to/.append style={
      edge node={node [sloped, allow upside down, auto=false]{$#1$}}}
  }
}

\title{Homotopy lifting, asymptotic homomorphisms, and traces}

\author{Tatiana Shulman}

\begin{document}

%\author{}
%\address{Department of Mathematical Sciences, UNiversity of Gothenburg, Chalmers tvärgata 3, 412 96 Gothenburg, Sweden}
%\email{tatshu@chalmers.se}

%\author{}
%\address{}
%\email{}

%\title{........}

%\subjclass[2010]{Primary 46L05; Secondary  20E26,  46L55}

%\keywords{}
\maketitle

\begin{abstract}  The following homotopy lifting theorem is proved: Let $\phi, \psi: B \to D/I$ be homotopic $\ast$-homomorphisms and suppose
$\psi$ lifts to a (discrete) asymptotic homomorphism. Then $\phi$ lifts to a (discrete)
asymptotic homomorphism. Moreover the whole homotopy lifts.

\noindent We also prove a  completely positive version of this theorem  and a version where $\phi$ is replaced by an asymptotic homomorphism.
These are the first homotopy lifting theorems that hold for all separable C*-algebras.

We obtain a lifting characterization  of  several important properties of C*-algebras and use them together with our homotopy lifting theorems to get the following applications:

\begin{itemize}

 \item MF-property is homotopy invariant;

\item If either $A$ or $B$ is exact, $A$ is homotopy dominated by $B$ and all amenable traces on $B$ are quasidiagonal, then all
amenable traces on $A$ are quasidiagonal;

\item If a C*-algebra $A$ is homotopy dominated by a nuclear C*-algebra $B$ and all (hyperlinear) traces on $B$ are MF, then all hyperlinear traces
on $A$ are MF.

\item Some of the extension groups introduced by Manuilov and Thomsen coincide.

\item The C*-algebra $qA$ from Cuntz's picture of KK-theory is always quasidiagonal.

\item Every homotopy symmetric C*-algebra is MF. 

\end{itemize}

\end{abstract}

\tableofcontents

\bigskip

\bigskip
\section{Introduction}

Homotopy invariants are among the most important tools in C*-algebra theory.   For instance, both K-theory and KK-theory are homotopy invariant in a very natural way. It also sometimes occurs that a property of C*-algebras does not seem to have anything topological in its nature, yet nevertheless turns out to be homotopy invariant.

 An example of such a property is quasidiagonality, as was discovered by Voiculescu \cite{Voiculescu}. Quasidiagonality is one of the central notions in C*-algebra theory and is closely related with several other approximation and embedding properties. These include  the MF-property, i.e., embeddability into $\frac{\prod M_n}{\oplus M_n}$, as well as approximation properties for traces. Following  the seminal paper \cite{TWW} such trace approximation properties have come to play a crucial role in the classification program for simple nuclear C*-algebras.
 In contrast to quasidiagonality it is not known whether all these properties are homotopy invariant.

 \medskip

 In this paper  we develop a novel approach to  homotopy invariance questions  by proving  homotopy lifting theorems that hold for all separable C*-algebras.

Homotopy lifting is a noncommutative analogue of homotopy extension which is fundamental in topology. It requires that if two morphisms are homotopic and one of them lifts, then the whole homotopy lifts. %IHomotopy lifting property is fundamental in topology. n the noncommutative setting homotopy extension is replaced by homotopy lifting that requires that if two morphisms are homotopic and one of them lifts, then the whole homotopy lifts.

%In this paper we prove new homotopy lifting theorems for C*-algebras and apply them to study  behavior of several C*-algebraic properties under homotopies.

Our homotopy lifting theorems involve asymptotic homomorphisms  -- a central object of E-theory of Connes and Higson, i.e., families of maps $\phi_{\lambda}$, $\lambda\in \Lambda$, that are asymptotically multiplicative, asymptotically linear etc. When $\Lambda$ is discrete, such families are usually called  discrete asymptotic homomorphisms.   By lifting of an (asymptotic) homomorphism $\phi_{\lambda}$ to an asymptotic homomorphism $\Phi_{\lambda}$ we will mean that  for each $\lambda$, $\Phi_{\lambda}$ is a lift of $\phi_{\lambda}$, following Manuilov and Thomsen (see e.g. \cite{Manuilov}). One can also consider a weaker, asymptotic instead of precise,  notion of lifting of asymptotic homomorphisms as suggested by Carrion and Schafhauser \cite{CSch} (see Preliminaries for precise definition).

\medskip

{\bf Theorem.} {\it (Theorem \ref{DiscreteHomotopyLifting}) Let $\phi, \psi: B \to D/I$ be homotopic homomorphisms and suppose $\psi$ lifts  to a (discrete) asymptotic homomorphism.
Then $\phi$ lifts  to a (discrete) asymptotic homomorphism. Moreover the whole homotopy lifts.}

\medskip

Our next result is a completely positive version of the homotopy lifting theorem ({\bf Theorems \ref{cpHomotopyLiftingCor1} and \ref{cpHomotopyLiftingCor2}}) that will be our main tool when working with properties involving completely positive maps, such as quasidiagonality and quasidiagonal traces.  We formulate it here in the following form.

\medskip

{\bf Theorem.} {\it Suppose $A$ is homotopy equivalent to $B$ and each $\ast$-homomorphism from $B$ to $D/I$ that has a cp lift  lifts to a cp (discrete) asymptotic homomorphism. Then each $\ast$-homomorphism from $A$ to $D/I$ that has a cp lift  lifts to a cp (discrete) asymptotic homomorphism.}

\medskip

The next result is a homotopy lifting theorem where one of two morphisms is an asymptotic homomorphism.

\medskip

{\bf Theorem.}{ (\it Theorem \ref{DiscreteHomotopyLiftingAsHom}) Let $\phi_{\lambda}: B \to A$, $\lambda\in [1, \infty)$, be an asymptotic homomorphism and $\psi: B \to A$ be a $\ast$-homomorphism that is homotopic to $\phi_{\lambda}, \lambda\in \Lambda$. Suppose $\psi$ lifts  to an asymptotic homomorphism. Then $\phi_{\lambda}$, $\lambda\in \Lambda$, lifts  to an asymptotic homomorphism.
 Moreover the whole homotopy lifts.}

\medskip

The key technical innovation underlying these results is the use of mapping cylinders in the study of homotopy lifting, together  with a new presentation of mapping cylinders by homogeneous relations, which turns out to be particularly well suited to lifting arguments.

\medskip

Our homotopy lifting theorems in fact hold not only for lifting but also for asymptotic lifting (e.g. in the sense of Carrion and Schafhauser).
These are the first homotopy lifting theorems that hold for all separable C*-algebras.

Prior to this paper, two homotopy lifting theorems were known in C*-theory.
Blackadar's homotopy lifting theorem states that if $A$ is a separable semiprojective C*-algebra, $\phi, \psi: A \to B/I$ are homotopic $\ast$-homomorphisms and $\psi$ lifts to a $\ast$-homomorphism, then  $\phi$ and the whole homotopy lift \cite{Blackadar}. This is a noncommutative analogue of Borsuk's Homotopy Extension Theorem for absolute neighborhood retracts.
 The other homotopy lifting theorem was obtained recently by Carrion and Schafhauser \cite{CSch}. It states that if a separable C*-algebra $A$ is an inductive limit of semiprojective C*-algebras, $\phi_t, \psi_t: A \to B/I$ are homotopic  asymptotic homomorphisms and $\psi_t$ lifts asymptotically, then  $\phi_t$ and the whole homotopy lift asymptotically \cite{CSch}.
%Asymptotic lifting is   a weaker form of lifting (for precise meaning of that see Preliminaries).

It is interesting to compare  our homotopy lifting theorem for asymptotic homomorphisms with Carrion-Schafhauser homotopy lifting theorem.
 Carrion and Schafhauser allow both morphisms to be asymptotic homomorphisms, but $A$ must be an inductive limit of semiprojective C*-algebras.  We require one of the morphisms to be an actual homomorphism but instead we have arbitrary $A$ and lifting instead of asymptotic lifting.   For example our theorem implies that any asymptotic homomorphism from a cone lifts, while Carrion-Schafhauser theorem says only that it lifts asymptotically.

%{\bf To mention taht they do not have a cp version!!!!!!!!!!!!!!!!!!!!!!!!!!!!!!!}

%The crucial ingredient of the proofs of our theorems above is   a nice  structure of  mapping cylinders, or, more precisely, the fact that they admit a presentation involving homogeneous relations. Homogeneous relations were introduced  in \cite{LoringShulman}. It was proved there that any cone C*-algebra admits a presentation involving only homogeneous relations, which was used to prove that any cone C*-algebra is inductive limit of projective C*-algebras. This approach was generalized to mapping cylinders by Thiel \cite{Thiel}. His ...... MAYBE to add all this!!!!!!!!!!!
\medskip

Liftings of (asymptotic) homomorphisms to asymptotic homomorphisms were first considered and proved useful by Manuilov and Thomsen in their series of works on extensions of C*-algebras. Here we prove that lifting to asymptotic homomorphisms is very useful for studying other important concepts as well.   We show  that the important notions of quasidiagonality and  MF property are in fact lifting properties, where lifting is meant to be the lifting to discrete asymptotic homomorphisms ({\bf Theorem  \ref{characterizationQD}, Theorem \ref{characterizationMF}}). This is somewhat surprising, since the MF-property is by definition an embedding property and quasidiagonality is rather an approximation property, so both do not look like they have anything to do with liftings at first glance. We also prove that the approximation properties for traces -- the properties of a trace of being quasidiagonal and MF, respectively -- admit a characterization in terms of liftings to discrete asymptotic homomorphisms ({\bf Proposition \ref{propositionMFtrace}}).

These characterizations together with our homotopy lifting theorems allow to study invariance of all these C*-algebraic properties under homotopies.

The first of the applications deals with the MF-property.

%is misterious !!!!!!! to write in second version!!!!!.....

\medskip

{\bf Theorem.} {\it (Theorem \ref{MFhominv}) If $A$ is homotopically dominated by $B$, and $B$ is MF, then $A$ is also MF.  In particular, MF property is homotopy invariant.}

\medskip

Voiculescu's result that quasidiagonality is invariant under homotopies  is also a consequence of our homotopy lifting theorem ({\bf Corollary \ref{QDhominv}}).

The question of whether all amenable traces are quasidiagonal is a famous open question (see \cite{BrownTraces}). By the celebrated result of Tikuisis-White-Winter  \cite{TWW} (generalized in   \cite{Gabe} and \cite{Schafhauser} to the class of all exact C*-algebras)  every faithful, amenable trace on an exact C*-algebra satisfying the UCT  is
quasidiagonal.
  Brown, Carrion and White  proved that any amenable trace on a cone C*-algebra is quasidiagonal \cite{BCW} (see \cite{sectionsCones} for another proof). This result has a homotopy flavour since cones are contractible.    In \cite{Neagu} Neagu specifically raised a question of whether the property that all amenable
traces are quasidiagonal is homotopy invariant. He proved that if $A$ is a separable exact C*-algebra with a faithful amenable trace, $A$ is homotopy dominated
by a C*-algebra $B$ and all amenable traces on $B$ are quasidiagonal, then all
amenable traces on $A$ are quasidiagonal.
The following theorem covers the aforementioned results from \cite{BCW} and \cite{Neagu} as particular cases.

\medskip

{\bf Theorem.} {\it (Theorem \ref{qdTracesExact}, Corollary \ref{BExact}) If either $A$ or $B$ is exact, $A$ is homotopy dominated by $B$ and all amenable traces on $B$ are quasidiagonal, then all
amenable traces on $A$ are quasidiagonal.}

\medskip

Another open problem is whether every hyperlinear trace is MF. Examples of C*-algebras with all hyperlinear traces being MF include many crossed product C*-algebras \cite{RainoneSch}. In the recent preprint  \cite{sectionsCones}  the author showed that any hyperlinear trace on any cone C*-algebra is MF. Here we get a wide generalization of the latter result.

\medskip

{\bf Theorem.} {\it (Corollary \ref{BNuclear}) If a C*-algebra $A$ is homotopy dominated by a nuclear C*-algebra $B$ and all (hyperlinear) traces on $B$ are MF, then all hyperlinear traces
on $A$ are MF.}

\medskip

One can also replace nuclearity by Hilbert-Schmidt stability ({\bf Corollary \ref{HilbertSchmidtStable}}).

 The next application deals with extension groups. It was proved by Manuilov and Thomsen that several extension groups $Ext_{\ast\ast}(SA, B)$ are equal \cite{MT15} and they asked whether in their result one can get rid of the suspension. We prove here that some of these equalities can be unsuspended ({\bf Theorem \ref{ExtCoincide}}).
 
 One further application of our homotopy lifting theorems concerns homotopy symmetric C*-algebras. Recall that E-theory was introduced by Connes and Higson in their celebrated paper   \cite{ConnesHigson}. It is defined as the set of homotopy classes of asymptotic homomorphisms from  $SA$ to $SB\otimes K$: $E(A, B) = [[SA, SB\otimes K]]$. In \cite{DadarlatLoring}, Dadarlat and Loring investigated when E-theory admits an unsuspended description. They proved that homotopy symmetric C*-algebras -- that is, C*-algebras for which $id_A$ is invertible  in the semigroup $[[A, A\otimes K]]$ -- are precisely those   for which E-theory can be unsuspended. 
 %More precisely, for every C*-algebra $B$, $E(A, B) = [[A, B\otimes K]]$ if and only if $A$ is homotopy symmetric.

In \cite{DadarlatPennig} Dadarlat and Pennig proved that nuclear homotopy symmetric C*-algebras are quasidiagonal.  We prove here that every homotopy symmetric C*-algebra is MF ({\bf Theorem \ref{homsymm}}). Since for nuclear C*-algebras the MF property is equivalent to quasidiagonality, our theorem recovers the aforementioned result of Dadarlat and Pennig.

Our last application deals with a C*-algebra that plays an important role in KK-theory. In \cite{Cuntz} Cuntz has introduced an alternative, more algebraic, description of the bifunctor $KK(A, B)$  based on the use of the C*-algebra $qA$ associated with a C*-algebra $A$.
In Cuntz's picture of KK-theory elements of $KK(A, B)$ are represented by homotopy classes of $\ast$-homomorphisms from $qA$ to  $B\otimes \mathcal K$, where $K$ is the C*-algebra of all compact operators on a Hilbert space.

We prove  here that  the C*-algebra $qA$ is always quasidiagonal, independently on what $A$ is  ({\bf Theorem \ref{qA}}).

\bigskip

\noindent {\bf Acknowledgments. } The author is grateful to Dominic Enders for  numerous discussions and very useful comments  on a draft of this paper.
  The author is partially supported by a grant from the Swedish Research Council.

\section{Preliminaries}

{\bf Asymptotic homomorphisms. }

\medskip

{\bf Definition} (\cite{ConnesHigson}). An {\it asymptotic homomorphism} from $A$ to $B$ is a family of maps $(f_{\lambda})_{\lambda\in [0,\infty)}: A \to B$ satisfying the following properties:

\medskip

(i) for any $a\in A$, the mapping $[0, \infty) \to B$ defined by the rule $\lambda \to f_{\lambda}(a)$  is continuous;

(ii) for any $a, b\in A$ and $\mu_1, \mu_2\in \mathbb C$, we have

\begin{itemize}
\item  $\lim_{\lambda\to \infty} \|f_{\lambda}(a^*) - f_{\lambda}(a)^*\| = 0;$
\item  $\lim_{\lambda\to \infty} \|f_{\lambda}(\mu_1a + \mu_2b) - \mu_1f_{\lambda}(a) - \mu_2f_{\lambda}(b)\| = 0;$
\item $\lim{\lambda\to \infty}  \|f_{\lambda}(ab) - f_{\lambda}(a)f_{\lambda}(b)\| = 0.$
\end{itemize}

We will call $(f_{\lambda})_{\lambda\in \Lambda}: A \to B$, where $\Lambda$ is a directed set,   a {\it discrete asymptotic homomorphism} if the condition (ii) above is satisfied and for each $a\in A$ one has $\sup_{\lambda} \|f_{\lambda}(a)\|< \infty$. (For usual asymptotic homomorphisms the last condition holds automatically, see \cite{ConnesHigson} or \cite{Dadarlat2}). In this paper discrete asymptotic homomorphisms mostly will be indexed by $\Lambda = \mathbb N$.

Two (discrete) asymptotic homomorphisms $(f_{\lambda})_{\lambda\in [0, \infty)}, (g_{\lambda})_{\lambda\in \Lambda}: A \to B$ are {\it equivalent} if for any $a\in A$, we have $\lim_{\lambda\to \infty} \|f_{\lambda}(a)- g_{\lambda}(a)\| = 0.$

Two (discrete) asymptotic homomorphisms $(f_{\lambda})_{\lambda\in \Lambda}, (g_{\lambda})_{\lambda\in \Lambda}: A \to B$ are {\it homotopy equivalent} if there exists an asymptotic homomorphism $(\Phi_{\lambda})_{\lambda\in \Lambda}: A \to B\otimes C[0, 1]$ such that $ev_0\circ \Phi_{\lambda} = f_{\lambda}, \; ev_1\circ \Phi_{\lambda} = g_{\lambda},$ $\lambda\in \Lambda$.

A (discrete) asymptotic homomorphism $(f_{\lambda})_{\lambda\in \Lambda}$ is {\it equicontinuous} if for any $a_0\in A$ and $\epsilon >0$ there is a $\delta>0$ such that $\|f_{\lambda}(a) - f_{\lambda}(a_0)\|<\epsilon$ for every $\lambda\in \Lambda$ whenever $\|a-a_0\|\le\delta$.

\medskip

Following Manuilov and Thomsen we will say that a $\ast$-homomorphism $f: A\to B/I$ (an asymptotic homomorphism $f_t: A \to B/I, t\in [1, \infty)$, respectively) {\it lifts to an asymptotic homomorphism } $\phi_t$, $t\in [1, \infty)$,  if $q\circ \phi_t = f$, for each $t$ ($q\circ \phi_t = f_t$, for each $t$, respectively).

\medskip

 In \cite{CSch} Carrion and Schafhauser considered a weaker notion of lifting of asymptotic homomorphisms which we will call here {\it an asymptotic lifting}.  Namely, a $\ast$-homomorphism $f: A\to B/I$ (an asymptotic homomorphism $f_t: A \to B/I, t\in [1, \infty)$, respectively) {\it  asymptotically lifts to an asymptotic homomorphism}  $\phi_t$ $t\in [1, \infty)$, if $\lim_{t\to \infty} q\circ \phi_t(a) = f(a)$, for any $a\in A$ ($\lim_{t\to \infty} \|q\circ \phi_t(a) - f_t(a)\| = 0$, for any $a\in A$, respectively).

  We define lifting and asymptotic lifting of discrete asymptotic homomorphisms similarly.

%\begin{lemma}\label{AsMultiplicativeFamily} Suppose $f_{\lambda}: C^*\langle \mathbf x\;|\;\mathbf R\rangle \to B$, $\lambda\in \Lambda$, is a family of maps such that $$\lim_{\lambda\to\infty}\|\mathbf R(f_{\lambda}(\mathbf x))\|= 0.$$ Then the family $f_{\lambda}, \lambda\in \Lambda$, is asymptotically multiplicative, asymptotically linear and asymptotically self-adjoint.\end{lemma}
%\begin{proof} Let $\pi: \prod_{\lambda} B \to \prod_{\lambda} B/\bigoplus_{\lambda} B$ be the canonical surjection. Then $\pi\circ(f_{\lambda})_{\lambda\in \Lambda}$ is a $\ast$-homomorphism. Therefore, for any $a, b\in C^*\langle \mathbf x\;|\;\mathbf R\rangle$,$$\pi\left(\left(f_{\lambda}(ab) - f_{\lambda}(a)f_{\lambda}(b)\right)_{\lambda\in \Lambda}\right)=0$$ that means that $$\lim_{\lambda\to\infty}\|f_{\lambda}(ab) - f_{\lambda}(a)f_{\lambda}(b)\|= 0.$$ Similarly one shows asymptotic linearity and self-adjointness.\end{proof}

We will call a not necessarily linear map {\it positive} if  it sends positive elements to positive elements.

%$\prod M_n/\oplus_{2, \omega} M_n$. We denote the trace .... on it by $tr$.

\bigskip

{\bf Homogeneous relations.}

\medskip

We will say that a non-commutative $\ast$-polynomial $p(x_1, \ldots, x_n)$ is {\it $d$-homogeneous} if $p(tx_1, \ldots, tx_n) = t^d p(x, y)$
for all real scalars $t$.  We call $d$  the {\it degree of homogeneity} of $p$.

For an NC $\ast$-polynomial $p(x_1, \ldots, x_n)$,   its {\it homogenization} $\tilde p(h, x_1, \ldots, x_n)$ is the homogeneous NC $\ast$-polynomial derived from $p$ by padding monomials on the left with
various powers of $h$. For example, if $$p(x_1, x_2, x_3) = x_1^4x_3 - x_2^*x_1 + x_1^*,$$ then
$$\tilde p (h, x_1, x_2, x_3) = x_1^4x_3 - h^3x_2^*x_1 + h^4x_1^*.$$

Let $A$ be a separable unital C*-algebra given by presentation
\begin{equation}\label{PresentationATAMS}A = \left\langle x_1, x_2, \ldots\;|\; -c_i\le x_i\le c_i,\; R_k(x_1, x_2, \ldots) =0, i, k\in \mathbb N\right\rangle.\end{equation}  Here and everywhere throughout this paper each
NC $\ast$-polynomial depends only on finitely many variables. By \cite[Lemma 7.3]{LoringShulman} any separable C*-algebra is of the form (\ref{PresentationATAMS}).  In \cite[Lemma 7.1]{LoringShulman} the following presentation  for the cone $CA$  was found\footnote{ In \cite[Lemma 7.1]{LoringShulman} there was assumed that the relations of $A$ do  not  have free term, but in fact this requirement can be omitted. }:
\begin{multline}\label{ConePresentationTAMS}  CA =  C^*\left\langle  h, x_1, x_2, \ldots \; \bigg |\; \begin{array}{c} 0\le h\le 1, -c_ih\le x_i\le c_ih, \; [h, x_i] = 0, \\ \tilde R_k(h, x_1, x_2, \ldots) = 0, \; i, k \in \mathbb N  \end{array} \right\rangle \end{multline}

Thus the cone has a presentation where every relation is homogeneous.

We note that if $A$ is non-unital, then, as the proof of \cite[Lemma 7.1]{LoringShulman} shows, all the elements $x_i$ and $x_ih^k$, $k\in \mathbb N$,  belong to $CA$ and generate it.

\medskip

We use notation $A^+$  for the minimal unitization of $A$. Throughout this paper $A$ and $B$ are mostly  separable with some exceptions.

\medskip

{\bf Traces}

\medskip

Recall that a trace $\tau$ on a C*-algebra $A$ is {\it amenable} if there is a sequence of   ccp maps $\phi_n: A \to M_{k_n}$, $n\in \mathbb N$, such that
$$ \lim_{n\to \infty} \|\phi_n(ab)- \phi_n(a)\phi_n(b)\|_2=0 \;\; \text{and} \;\; \lim_{n\to \infty}|\tau(a)- tr \phi(a)|=0,$$ for any $a, b\in A$.

A trace $\tau$ is {\it quasidiagonal} if there is a sequence of   ccp maps $\phi_n: A \to M_{k_n}$, $n\in \mathbb N$, such that
$$ \lim_{n\to \infty} \|\phi_n(ab)- \phi_n(a)\phi_n(b)\|=0 \;\; \text{and} \;\; \lim_{n\to \infty}|\tau(a)- tr \phi(a)|=0,$$  for any $a, b\in A$.

If in the the definition of an amenable trace we drop the requirement that  $\phi_n$'s are ccp, then we obtain the notion of a hyperlinear trace. Precisely, a trace $\tau$ is {\it hyperlinear} if there is a sequence of  maps $\phi_n: A \to M_{k_n}$, $n\in \mathbb N$, such that
\begin{multline*}  \lim_{n\to \infty} \|\phi_n(ab)- \phi_n(a)\phi_n(b)\|_2=0 \;\; \lim_{n\to \infty}\|\phi_n(\lambda a+\mu b)- \lambda\phi_n(a)-\mu \phi_n(b)\|_2 =0, \\ \lim_{n\to \infty} \|\phi_n(a^*) - \phi_n(a)^*\|_2=0,\;\; \sup_{n\in \mathbb N} \|\phi_n(a)\|<\infty, \;\text{and} \;\; \lim_{n\to \infty}|\tau(a)- tr \phi(a)|=0,\end{multline*} for any $a, b\in A$, $\lambda, \mu\in \mathbb C$.

 \begin{remark}\label{hypTraceReformulation} Equivalently, one can say that $\tau$ is {\it hyperlinear}  if $\tau = tr\circ f$, for some $\ast$-homomorphism $f: A \to \prod M_{k_n}/\oplus_2 M_{k_n}$, for some $k_n$'s,  where $\oplus_2 M_{k_n}$ is the ideal of all sequences that converge to zero in the 2-norm and $tr$ is a  trace on $\prod M_{k_n}/\oplus_2 M_{k_n}$ defined by the formula
 $tr([(T_n)_{n\in \mathbb N}]) = \lim_{n\to \omega} tr T_n$, where $\omega$ is some non-trivial ultrafilter on $\mathbb N$. One also can consider the ideal $\oplus_{2, \omega} M_{k_n} = \{(T_n)_{n\in \mathbb N}\;|\; \lim_{n\to \omega} \|T_n\|_2 = 0\}$ of $\prod M_{k_n}$. The C*-algebra $\prod M_{k_n}/\oplus_{2, \omega} M_{k_n}$, called  the tracial ultraproduct of matrix algebras,  is a von Neumann algebra with the faithful trace $tr([(T_n)_{n\in \mathbb N}]) = \lim_{n\to \omega} tr T_n$. It is easy to show that a trace $\tau$ is hyperlinear if and only if $\tau = tr\circ f$, for some $\ast$-homomorphism $f: A \to \prod M_{k_n}/\oplus_{2, \omega} M_{k_n}$ and some $k_n$'s.

\noindent One can also replace  $\prod M_n/\oplus_{2, \omega} M_{n}$ by $\mathcal R^{\omega}$.
\end{remark}

If in the definition of a quasidiagonal trace one drops the requirement that $\phi_n$'s are ccp, one obtains the notion of an MF-trace. Precisely,
a trace $\tau$ is {\it MF} if there is a sequence of  maps $\phi_n: A \to M_{k_n}$, $n\in \mathbb N$, such that
\begin{multline*}  \lim_{n\to \infty} \|\phi_n(ab)- \phi_n(a)\phi_n(b)\|=0 \;\; \lim_{n\to \infty}\|\phi_n(\lambda a+\mu b)- \lambda\phi_n(a)-\mu \phi_n(b)\| =0, \\ \lim_{n\to \infty} \|\phi_n(a^*) - \phi_n(a)^*\|=0,\;\; \sup_{n\in \mathbb N} \|\phi_n(a)\|<\infty, \; \text{and} \;\; \lim_{n\to \infty}|\tau(a)- tr \phi(a)|=0,\end{multline*} for any $a, b\in A$, $\lambda, \mu\in \mathbb C$.

\begin{remark} In the definition of an MF-trace  one can additionally require the maps $\phi_n$ to be $\ast$-linear   \cite[Prop. 2.2]{RainoneSch}. The same argument  works also for  hyperlinear traces.
%Indeed the homomorphism $f: A \to \prod M_n/\oplus M_n$ can be lifted to a linear (not necessarily continuous) map by sending any element of a Hamel basis of $A$ to any its preimage and extending it linearly to $A$.  Then $\phi$ can be defined as the $n$-th coordinate of this lift.
\end{remark}

Clearly every quasidiagonal trace is amenable and every MF-trace is hyperlinear. Whether the converse implications hold is an open problem.

% Daje esli b u nas ne bylo unitalnoy versii glavnoy teoremy, teorema pro sledy vse ravno poluchalas by, perehodya ot phi_n k phi_n^+  iz C(A^+) v multiplicatory M_n = M_n.

%One of biggest open problems is whether every amenable trace is quasidiagonal. TWW proved.........

\medskip

\section{A presentation for a mapping cylinder}

Suppose $\psi: B\to A$ is a $\ast$-homomorphism. Recall that the mapping cylinder $Z_{\psi}$ is the C*-algebra
$$Z_{\psi} = \{(\eta, b)\;|\; \eta\in C[0, 1]\otimes A, b\in B, \eta(0) = \psi(b)\} \subset \left(C[0, 1]\otimes A\right)\oplus B.$$

There is a natural embedding of $B$ into $Z_{\psi}$ defined by $$b\mapsto (1\otimes \psi(b), b),$$ $b\in B$. So everywhere below we consider $B$ as a C*-subalgebra of $Z_{\psi}$. There is also a canonical embedding of $CA$ into $Z_{\psi}$ defined by $$\eta\mapsto (\eta, 0),$$ $\eta \in CA$.  There is a natural short exact sequence $$0\to CA\to Z_{\psi} \to B\to 0$$ that splits via the embedding of $B$ into $Z_{\psi}$ described above.

%Let $B^+$ denote the forced unitization of $B$ and let $\psi^+: B \to A^+$ denote the composition of $\psi$ with the inclusion $A\subset A^+$.

In this chapter we will assume that $B$ and $A$ are separable.

It was discovered by Thiel \cite{Thiel} that  a mapping cylinder has a presentation involving  homogeneous relations. He used it to prove that being an inductive limit of semiprojective C*-algebras is a homotopy invariant property.
Here we give another presentation for a mapping cylinder that we find somewhat more transparent than the one in \cite{Thiel}. It  also involves homogeneous relations which will be crucial for the rest of the paper.

\medskip

We will use notation $\mathbf x = \{x_1, x_2, \ldots\}$ for the set of generators (also $\mathbf X$ for a tuple $\{X_1, X_2, \ldots\}$ in some C*-algebra $A$) and $\mathbf R = \{R_1, R_2, \ldots\}$ for the set of relations. In this paper all the relations will be either noncommutative $\ast$-polynomial equalities and inequalities or inequalities for the norms of generators.   We will write $C^*\langle \mathbf x\;|\;\mathbf R\rangle $ for the corresponding universal C*-algebra. We will write $\mathbf R(\mathbf X) =0$ meaning that $X_1, X_2, \ldots$ satisfy all the relations.

\begin{lemma}\label{PresentationA} Let $B = C^*\langle \mathbf x \;|\; -\mathbf c \le \mathbf x\le c, \;  \mathbf p(\mathbf x) = 0, \mathbf c>0\rangle$
and let $\psi: B \to A$ be a $\ast$-homomorphism. Then $A$ can be written as the universal C*-algebra
$$A = \langle \mathbf z, \mathbf y\;|\; -\mathbf d \le \mathbf z\le \mathbf d,  -\mathbf e \le \mathbf y\le \mathbf e, \;  \mathbf p(\mathbf z) = 0, \mathbf q(\mathbf z, \mathbf y) =0, \mathbf d>0, \mathbf e>0\rangle ,$$
such that $\psi(\mathbf x) = \mathbf z$.
\end{lemma}
\begin{proof} The proof goes along the lines of \cite[Lemma 7.3]{LoringShulman}. Let $\mathbf x = \{x_1, x_2, \ldots\}$ and $\mathbf p = \{p_1, p_2, \ldots\}$. To the collection $z_i: = \psi(x_i)\in A$, $i\in \mathbb N$, we add more elements of $A_{sa}$ to obtain a dense countable subset of $A_{sa}$. We apply to this sequence all polynomials over $\mathbb F= \mathbb Q + i \mathbb Q$ in countably many variables. This results in a countable, dense $\mathbb F$-$\ast$-subalgebra $\mathcal D$ of $A$. Enumerate $\mathcal D \setminus \{z_1, z_2, \ldots\}$ as $y_1, y_2, \ldots$. The algebraic
operations for $\mathcal D$ can be encoded in $\ast$-polynomial relations. For example, if $\alpha y_j = y_k$ for
some $\alpha\in \mathbb F$, then we use the relation $\alpha y_j - y_k = 0$. If $(z_k^2y_j)^*(z_k^2y_j)  = y_i$, then we use the relation
 $(z_k^2y_j)^*(z_k^2y_j)  - y_i = 0$, and so forth. We now add to these relations the relations $p_i(z_1, z_2, \ldots) = 0$ (which are already included if $p_i$'s have rational coefficients) and the relations $-d_i\le z_i\le d_i$, $-e_i \le y_i\le e_i$, $i\in \mathbb N$, where $d_i$ is the norm of the element $z_i$ in $A$ and $e_i$ is the norm of the element $y_i$ in $A$. Then the proof of \cite[Lemma 7.3]{LoringShulman}   goes without change to show that $A$ is universal for these relations.
\end{proof}

For an NC *-polynomial $p$ in the noncommuting variables $\mathbf x, \mathbf x^*$ we define its homogenization $\tilde p$  to be the NC polynomial derived from $p$ by padding monomials on the left with
various powers of a new variable $h$ so that $\tilde p$ is homogeneous. For example, if $p(\mathbf x) = x_1^* - x_2^3$, then $\tilde p(\mathbf x, h)= h^2x_1^* - x_2^3$.

\begin{theorem}\label{presentation} Let $A$ be a unital C*-algebra, $\psi: B \to A$ a $\ast$-homomorphism, $B = C^*\langle \mathbf x \;|\; -\mathbf c \le \mathbf x\le \mathbf c, \;  \mathbf p(\mathbf x) = 0\rangle.$ Write $B$ and $A$ as in Lemma \ref{PresentationA}. Then the mapping cylinder $Z_{\psi}$ has presentation

\[Z_{\psi}\cong C^* \left\langle \mathbf x', \mathbf z', \mathbf y', h \bigg | \begin{array}{c} 0\le h\le 1, \;-\mathbf c \le \mathbf x' \le \mathbf c,\; -\mathbf d h \le \mathbf z' \le \mathbf d h, \; -\mathbf e h \le \mathbf y' \le \mathbf e h, \\  h \mathbf x'-\mathbf x'h = h\mathbf z'-\mathbf z'h = h\mathbf y'-\mathbf y'h =0  \\ h\mathbf x' = \mathbf z',\\ \mathbf p(\mathbf x')=0, \\ \mathbf{\tilde p}(\mathbf z', h)=0, \; \mathbf{\tilde q}(\mathbf z', \mathbf y', h)=0 \end{array}\right\rangle.\]

The canonical embedding of $B$ into $Z_{\psi}$ sends $\mathbf x$ to $\mathbf x'$. The canonical embedding of $CA \cong \left\langle \mathbf z'', \mathbf y'', k \bigg | \begin{array}{c} 0\le k\le 1, \; -\mathbf d k \le \mathbf z'' \le \mathbf d k, \; -\mathbf e k \le \mathbf y'' \le \mathbf e k, \\   k\mathbf z''-\mathbf z''k = k\mathbf y''-\mathbf y''k =0,\\ \mathbf{\tilde p}(\mathbf z'', k)=0, \; \mathbf{\tilde q}(\mathbf z'', \mathbf y'', k)=0 \end{array}\right\rangle$ into $Z_{\psi}$ sends $\mathbf z''\mapsto \mathbf z', \mathbf y''\mapsto \mathbf y', k\mapsto h$.
\end{theorem}
\begin{proof} Let $\mathcal U$ be the universal C*-algebra above. We are going to prove that $\mathcal U \cong Z_{\psi}$.
Define  a $\ast$-homomorphism $\theta: \mathcal U \to Z_{\psi}$ on the generators of $\mathcal U$ by
$$\theta (\mathbf x') = (1\otimes \psi(\mathbf x), \mathbf x) = (1\otimes \mathbf z , \mathbf x),$$
$$\theta(\mathbf z') = (t\otimes \mathbf z, 0),$$
$$\theta(\mathbf y') = (t\otimes \mathbf y, 0),$$
$$ \theta(h) = (t\otimes 1_A, 0).$$

To check that $\theta$ is indeed a $\ast$-homomorphism we need to check that all the relations of $\mathcal U$ are satisfied.
All of them but the last two are clearly satisfied. So we check here the last two.

Write each polynomial $p$ in $\mathbf p = (p_1, p_2, \ldots)$ as
$p = \sum_{k=0}^N p_k$ with $p_k$ being homogeneous of degree $k$. Its homogenization is  $$\tilde p(\mathbf s_1, s_2) = \sum_{k=0}^N {s_2}^{N-k} p_k(\mathbf s_1).$$ We have
\begin{equation}\label{PresentationEq1} \tilde p(\theta(\mathbf z'), \theta(h)) = \left(\tilde p(t\otimes \mathbf z, t\otimes 1_A), 0\right),
\end{equation}  and for any $s\in [0, 1]$
\begin{multline*}\tilde p(t\otimes \mathbf z, t\otimes 1_A)(s) = \left(\sum_{k=0}^N p_k(t\otimes \mathbf z)\left(t^{N-k}\otimes 1_A\right)\right)(s)\\
= \sum_{k=0}^N p_k(s \mathbf z)  s^{N-k} 1_A = \sum_{k=0}^N s^k p_k(\mathbf z) s^{N-k}\\=s^{N}\sum_{k=0}^N p_k(\mathbf z)
 = s^{N} p(\mathbf z) =0.
\end{multline*}
Substituting it to (\ref{PresentationEq1}) we obtain that $\tilde p(\theta(\mathbf z'), \theta(h))=0$. In the same way one checks that
$\tilde q(\theta(\mathbf z'), \theta(\mathbf y'), \theta(h))=0$.

 The image of $\theta$ contains both $B$ and $CA$. Since $B$ and $CA$ generate $Z_{\psi}$, $\theta$ is surjective.

To prove that $\theta$ is injective we will show that each irreducible representation of $\mathcal U$ factorizes through $\theta$.
So let $\pi$ be an irreducible representation of $\mathcal U$. Since $h$ is a central element of $\mathcal U$, $\pi(h) = \lambda 1$, for some $\lambda$. We will consider separately the case of $\lambda=0$ and the case of $\lambda\neq 0$.

{\bf Case $\lambda =0$.} Then $\pi(h)=0$. The norm conditions on $\mathbf z', \mathbf y'$ imply that $$\pi(\mathbf z')=\pi(\mathbf y') =0.$$  We define a $\ast$-homomorphism $\sigma: B \to B(H)$ on the generators of $B$ by
$\sigma(\mathbf x) = \pi(\mathbf x')$.  Let $pr_B: Z_{\psi}\to B$ be the homomorphism given by $(f, b) \mapsto b$. Then $\pi = \sigma\circ pr_B\circ \theta.$ Indeed
\begin{multline*} \sigma\circ pr_B\circ \theta(\mathbf x') = \pi(\mathbf x'),  \\ \sigma\circ pr_B\circ \theta(\mathbf z') =0 =\pi(\mathbf z'), \\ \sigma\circ pr_B\circ \theta(\mathbf y') =0=\pi(\mathbf y').\end{multline*}

{\bf Case $\lambda\neq 0$.} Since $\pi(\mathbf z') = \pi(h\mathbf x')$, we have $\pi(\mathbf x') = \pi(\mathbf z')/{\lambda}.$ Define a $\ast$-homomorphism $\tau: A \to B(H)$  on the generators of $A$ by
$$\tau(\mathbf z) = \pi(\mathbf x') = \pi(\mathbf z')/{\lambda},$$
$$\tau(\mathbf y) = \pi(\mathbf y')/{\lambda}.$$ Let us check that all the relations of $A$ are satisfied. We can write $p(\mathbf z)= \sum_{i=1}^M p^{(i)}(\mathbf z)$, $q(\mathbf z, \mathbf y)= \sum_{i=1}^N q^{(i)}(\mathbf z, \mathbf y)$ with $p^{(i)}, q^{(i)}$ being homogeneous of degree $i$.  Then

\begin{multline*} p(\tau(\mathbf z))  = p(\pi(\mathbf z')/\lambda) = \sum_{i=1}^M \frac{p^{(i)}(\pi(\mathbf z'))}{\lambda^i} = \frac{1}{\lambda^M}
\sum_{i=1}^M p^{(i)}(\pi(\mathbf z'))\lambda^{M-i} \\ =  \frac{1}{\lambda^M}
\sum_{i=1}^M p^{(i)}(\pi(\mathbf z'))\pi(h)^{M-i} = \frac{1}{\lambda^M} \tilde p(\pi(\mathbf z'), \pi(h))=  0. \end{multline*}

In the same way we obtain $$q(\tau(\mathbf z), \tau(\mathbf y))  = \frac{1}{\lambda^N} \tilde q(\pi(\mathbf z'), \pi(\mathbf y'), \pi(h))=0. $$
Thus $\tau$ is well-defined.  Let $ev_{t}: Z_{\psi} \to A$ be the homomorphism defined by $(f, b) \mapsto f(t)$. It is straightforward to check that $\pi$ and $\tau\circ ev_{\lambda} \circ \theta$ coincide on the generators of $\mathcal U$. Therefore $\pi =\tau\circ ev_{\lambda} \circ \theta$.
\end{proof}

\section{Mapping cylinders and homotopy lifting}

The first lemma is very well known (see e.g. \cite{LoringBook}).

\begin{lemma}\label{folklore}
Let $\pi: B\to B/I$ be a surjective $\ast$-homomorphism.
For any approximate unit  $\{u_{\lambda}\}$ in $I$ and any $x\in B$, one has $\limsup \|x(1- u_{\lambda})\| = \|\pi(x)\|$.
\end{lemma}

The following lemma is essentially contained in \cite{LoringShulman}. We write out its proof explicitly here as it will be used very often in this paper.

\begin{lemma}\label{AsHom1}  Let $p (x_1, \ldots, x_N) $ be a homogeneous NC $\ast$-polynomial (more generally, $p$ can be of more variables and homogeneous only in $x_1, \ldots, x_N$). Let $I \lhd B$, $\{i_{\lambda}\}$ a quasicientral approximate unit of $I$ relative to $B$, and $\pi: B \to B/I$  the canonical surjection.
 Suppose $\pi(p(b_1, \ldots, b_N)) = 0$. Then $$\lim \|p(b_1(1-i_{\lambda}), \ldots, b_N(1-i_{\lambda}))\| =0.$$
\end{lemma}
\begin{proof} Let $d$ be the degree of homogeneity of $p$. By quasicentrality of  $\{i_{\lambda}\}$ we have $$ \|p(b_1(1-i_{\lambda}), \ldots, b_N(1-i_{\lambda}))\| \approx \|p(b_1, \ldots, b_N) (1- i_{\lambda})^d\|.$$ Since $\{ 1- (1- i_{\lambda})^d\}$ is itself a (quasicentral) approximate unit, by Lemma \ref{folklore}  we obtain
$$\limsup \|p(b_1(1-i_{\lambda}), \ldots, b_N(1-i_{\lambda}))\| = \| \pi(p(b_1, \ldots, b_N)) \| =0.$$
\end{proof}

We will also need the following lemma. We let $C_b([1, \infty), B)$ denote the C*-algebra of all bounded continuous $B$-valued functions on $[1, \infty)$ and let $C_0([1, \infty), B)$ be the ideal of all functions vanishing at infinity.

\begin{lemma}\label{ExtendingFromGenerators} (\cite[Lemma 7, Remark 8]{sectionsCones}) (i) Let $\phi: C^*\langle \mathbf x\;|\;\mathbf R\rangle \to B/I$ be a $\ast$-homomorphism and let $\mathbf X \in C_b([1, \infty), B)$  be such that
\begin{equation}\label{ExtendingFromGenerators1} \lim_{t\to\infty} \mathbf R(\mathbf X)(t) = 0\end{equation}
and
\begin{equation}\label{ExtendingFromGenerators1'} \pi(\mathbf X(t)) = \phi(\mathbf x), \end{equation}
for any $t\in [1,\infty)$.
Then there exists a contractive positive asymptotic homomorphism $f_t: C^*\langle \mathbf x\;|\;\mathbf R\rangle\to B$ such that $\pi\circ f_t = \phi$, for any $t\in [1,\infty)$, and $\lim_{t\to \infty} \|f_t(\mathbf x) - \mathbf X(t)\| = 0$.

\medskip

(ii) Let $p_1, p_2, \ldots$ be noncommutative $\ast$-polynomials.
 Let $B_0\subset B$ be a C*-subalgebra and suppose
$p_k(\mathbf X)\in C_b([0, \infty), B_0)$, for each $k\in \mathbb N$. Then the asymptotic homomorphism $f_t$ in (i) can be chosen with the additional property that $f_t(C^*(p_1(\mathbf x), p_2(\mathbf x), \ldots))\subset B_0$, for each $t$.

Similar statements hold in the case when the parameter is discrete.
\end{lemma}

%Suppose $$B = C^*\langle \mathbf x = (x_1, x_2, \ldots |\; -c_k\le x_k\le c_k, p_k(\mathbf x)=0, k\in \mathbf N\rangle.$$
  %It was discovered by Thiel [] that the mapping cylinder $Z_{\psi^+}$ has the following presentation: \begin{multline*}Z_{\psi^+} \cong C^*\langle -c_k\le x_k\le c_k, \; p_k(\mathbf x) = 0 , k\in \mathbb N,\\ 0\le h\le 1, [h, x_k] = 0, [h, y_k]=0, k\in \mathbb N,\\-d_jh \le y_j\le d_jh, j\in \mathbb N,\;\tilde q_l(\mathbf x, \mathbf y, h)=0, l\in \mathbf N\rangle, \end{multline*}where $\tilde q_l$ is obtained by homogenizing some polynomial $q_l(\mathbf x, \mathbf y) = \sum_{d=0}^{N_l} q_{l, d}(\mathbf x, \mathbf y)$, with $N_l\ge 1$ and each $q_{l, d}$ being $d$-homogeneous in $\mathbf y$,  that is $$\tilde q_l(\mathbf x, \mathbf y, h) = \sum_{d=0}^{N_l} h^{N_l-d}q_{l, d}(\mathbf x, \mathbf y).$$ In particular one sees that $\tilde q_l$'s are homogeneous in the subset $\{\mathbf y, h\}$ of variables, and this is the only thing that we need.

The following lemma is straightforward.

\begin{lemma}\label{continuousqau} Let $u_n, n\in \mathbb N,$ be a quasicentral approximate unit for $I\triangleleft B$. Let $u_t = (n+1-t)u_n + (t-n)u_n$, for $n\le t\le n+1$, $n\in \mathbb N$. Then the continuous path $u_t$, $t\in [1, \infty)$, is also a quasicentral approximate unit for $I\triangleleft B$.
\end{lemma}

\begin{theorem}\label{cylinder} Let $A$ and $B$ be separable C*-algebras and $\psi: B \to A$ a $\ast$-homomorphism. Let $f: Z_{\psi}\to D/I$ be a $\ast$-homomorphism. Suppose $f|_B$ lifts to a (discrete) asymptotic homomorphism. Then $f$ lifts to a (discrete) asymptotic homomorphism.
\end{theorem}
\begin{proof} 1) First assume that $A$ is unital. We will use the presentation of $Z_{\psi}$ from Theorem \ref{presentation}. Lift $f|_B$ to an asymptotic homomorphism $\tilde f_{\lambda}: B \to D$, $\lambda\in \Lambda$. Here $\Lambda=\mathbb N$ if we are interested in discrete asymptotic homomorphisms and $\Lambda = [0, \infty)$ for the continuous version.  Let us denote $\tilde f_{\lambda}(\mathbf x)$ by  $\mathbf X_{\lambda}$.
Lift $f(h)$ to $0\le H\le 1$. By Davidson's  two-sided order lifting theorem \cite[Lemma 2.1]{DavidsonLiftingPositive} we can lift $f(\mathbf y')$ and $f(\mathbf z')$ to $\mathbf Y'\in D$ and $\mathbf Z'\in D$ respectively, with $-\mathbf dH \le \mathbf Z'\le \mathbf dH$,  $-\mathbf eH \le \mathbf Y'\le \mathbf eH$.

Let $u_{\lambda}$ be a quasicentral approximate unit in $I$ relative to $D$ (for the continuous version we take   a quasicentral approximate unit forming a continuous path as in Lemma \ref{continuousqau}).
%Let $N_n$ be the number of variables that $p_k$ and $\tilde q_l$ depend of, when $k, l\le n$.
 Let  $$H_{\lambda} = (1-u_{\lambda})^{1/2}H(1-u_{\lambda})^{1/2},$$ $$\mathbf Z'_{\lambda} = (1-u_{\lambda})^{1/2}\mathbf Z'(1-u_{\lambda})^{1/2},$$
$$\mathbf Y'_{\lambda} = (1-u_{\lambda})^{1/2}\mathbf Y'(1-u_{\lambda})^{1/2}, $$. By Lemma \ref{AsHom1}
%and \ref{AsHom2} there is $\lambda_n$  such that
  we have
$$  \lim_{\lambda} \|\mathbf{\tilde p}\left(\mathbf Z'_{\lambda},  H_{\lambda}\right)\| = 0, $$
$$\lim_{\lambda}\|\mathbf{\tilde q}\left(\mathbf Z'_{\lambda},
\mathbf Y'_{\lambda}, H_{\lambda}\right)\|=0, $$
 $$\lim{\lambda}\|[H_{\lambda},  \mathbf X_{\lambda}]\|=0, \; \lim_{\lambda}\|[H_{\lambda},  \mathbf Z'_{\lambda}]\|=0, $$
 $$\lim_{\lambda}\|[H_{\lambda},  \mathbf Y'_{\lambda}]\|=0, \;
\lim_{\lambda}\|H_{\lambda} \mathbf X_{\lambda} -  \mathbf Z'_{\lambda} \|=0.$$
Since $\tilde f_{\lambda}$ is an asymptotic homomorphism,
$$\lim_{\lambda} \|\mathbf p(\mathbf X_{\lambda})\|=0.$$
We also have $$-\mathbf dH_{\lambda} \le \mathbf Z'_{\lambda}\le \mathbf dH_{\lambda},\;  -\mathbf eH_{\lambda} \le \mathbf Y'_{\lambda}\le \mathbf eH_{\lambda}.$$ So all the relations of $Z_{\psi}$ are approximately satisfied. Note that $\mathbf Z'_{\lambda}$, $\mathbf Y'_{\lambda}$ and $H_{\lambda}$ are lifts of $f(\mathbf z')$, $f(\mathbf y')$ and $f(h)$ respectively, for each $\lambda\in \Lambda$.
By Lemma \ref{ExtendingFromGenerators}, $f$ lifts to a (discrete) asymptotic homomorphism.

\medskip

2) Now assume $A$ is non-unital.   We can assume $f$ is surjective.
Let $\psi^+: B \to A^+$ be the composition of $f$ with the canonical embedding $A\to A^+$. Since  $Z_{\psi}$ is an essential ideal in $Z_{\psi^+}$, we have $Z_{\psi} \subset Z_{\psi^+}\subset M(Z_{\psi})$. By the NC Tietze Extension Theorem $f$ extends to a $\ast$-homomorphism $ f': M(Z_{\psi})\to M(D/I)$. Now we apply the arguments from 1) to
$\tilde f:=  f'|_{Z_{\psi^+}}: Z_{\psi^+}\to M(D/I).$

We have

a)The generators $\mathbf x'$ belong to $B\subset Z_{\psi}$ and generate it.

b)  The elements  $\mathbf y'$,  $\mathbf z'$, $\mathbf y'h^k$,  $\mathbf z' h^k$, where $k\in \mathbb N$,     belong to $CA\subset Z_{\psi}$ and generate it.

c) $Z_{\psi}$ is generated by $B$ and $CA$.

d) By the assumption, $\mathbf X_{\lambda}$ is in $D$. By the proof of \cite[Lemma 2.1]{DavidsonLiftingPositive}, since $D$ is an ideal in $M(D)$, $\mathbf Y'_{\lambda}$, $\mathbf Z'_{\lambda}$ in our construction  also can be chosen to be in $D$. Therefore $\mathbf Y'_{\lambda}H_{\lambda}^k$,  $\mathbf Z'_{\lambda} H_{\lambda}^k$, $k\in \mathbb N$,  are also in $D$.

\noindent  By Lemma \ref{ExtendingFromGenerators} (ii), $\tilde f$  lifts to a (discrete) asymptotic homomorphism $f_{\lambda}$   such that $f_{\lambda}|_{Z_{\psi}}$ lands in $D$.
\end{proof}

\medskip

\begin{lemma}\label{factorization} Let $\phi, \psi: B \to A$ be homotopic $\ast$-homomorphisms. Then $\phi$ factorizes through $Z_{\psi}$,

$$ \begin{tikzcd}  B \arrow{r}{\alpha}\arrow[bend right=30,  swap]{rr}{\phi}& Z_{\psi}\arrow{r}{\beta}& A \end{tikzcd}$$
meaning that there exist  $\ast$-homomorphisms $\alpha: B \to Z_{\psi}$ and $\beta: Z_{\psi}\to A$ such that $\phi = \beta\circ \alpha$.
Moreover $\beta|_B = \psi$.
\end{lemma}
\begin{proof} Let $\Phi: B\to A\otimes C[0, 1]$ be a homotopy between $\phi$ and $\psi$ with $$ev_1\circ\Phi = \phi, \;ev_0\circ\Phi=\psi.$$ Then we can define $\alpha: B \to Z_{\psi}$ by
$$\alpha(b) = (\Phi(b), b),$$ for any $b\in B$.
Define $\beta: Z_{\psi}\to A$ by
$$\beta((\xi, b)) = \xi(1),$$ for any $(\xi, b)\in Z_{\psi}$. Then $\beta\circ \alpha = \phi$.
As usual we consider $B$ as a C*-subalgebra of $Z_{\psi}$ via the embedding $b\mapsto (1\otimes \psi(b), b)$. We have
$$\beta|_B(b) = \beta((1\otimes \psi(b), b))= \psi(b).$$
\end{proof}

\begin{theorem}\label{DiscreteHomotopyLifting} Let $B$ be a separable C*-algebra. Let $\phi, \psi: B \to D/I$ be homotopic homomorphisms and suppose $\psi$ lifts  to a (discrete) asymptotic homomorphism.
Then $\phi$ lifts  to a (discrete) asymptotic homomorphism. Moreover the whole homotopy lifts.
\end{theorem}
\begin{proof} We use the notation and constructions of Lemma \ref{factorization}. By Lemma \ref{factorization} $\phi = \beta\circ \alpha$ with $\beta: Z_{\psi}\to D/I$ such that $\beta|_B = \psi$. Since $\psi$ lifts  to a (discrete) asymptotic homomorphism, by Theorem \ref{cylinder}  $\beta$ lifts to a (discrete) asymptotic homomorphism $\gamma_{\lambda}$, $\lambda\in \Lambda$. Then $\phi$ lifts to $\gamma_{\lambda}\circ\alpha$, $\lambda\in \Lambda$.

We now show that the whole homotopy $\Phi$ between $\phi$ and $\psi$ lifts.  For each $0\le s\le 1$ and $b\in B$ we define $\Gamma_{b, s}\in A\otimes C[0,1]$ by $$\Gamma_{b, s}(t) = \Phi(b)(st),$$ $t\in [0, 1]$. Since $\Gamma_{b, s}(0) = \Phi(b)(0) = \psi(b)$, we have $(\Gamma_{b, s}, b) \in Z_{\psi}$. Since the assignment $s\mapsto \Gamma_{b, s}$ is continuous, we can define a $\ast$-homomorphism $\Theta: B \to Z_{\psi}\otimes C[0, 1]$ by
$$\Theta(b)(s) = (\Gamma_{b, s}, b).$$ Then $\Phi = (\beta\otimes id_{C[0, 1]})\circ \Theta.$ Therefore  $(\gamma_{\lambda}\otimes id_{C[0, 1]})\circ \Theta,$ $\lambda\in \Lambda,$ is a homotopy that lifts $\Phi$.
\end{proof}

\bigskip

\begin{corollary}\label{HomInvGen} If $A$ and $B$ are homotopy equivalent (or just $A$ homotopically dominates $B$), and each $\ast$-homomorphism from $A$ to $D/I$ lifts to a (discrete) asymptotic homomorphism, then each $\ast$-homomorphism from $B$ to $D/I$ lifts to a (discrete) asymptotic homomorphism.
\end{corollary}
\begin{proof} Since $A$ homotopically dominates $B$, there exist $\ast$-homomorphisms $\alpha: A \to B$ and $\beta: B\to A$ such that $\alpha\circ \beta$ is homotopic to $id_B$. Let $f: B \to D/I$ be a $\ast$-homomorphism. The $\ast$-homomorphism $f\circ \alpha: A \to D/I$ lifts to a (discrete) asymptotic homomorphism $\gamma_{\lambda}: A \to D$, $\lambda\in \Lambda$. Then $f\circ \alpha\circ \beta$ lifts  to the (discrete) asymptotic homomorphism $\gamma_{\lambda}\circ \beta: A \to D$, $\lambda\in \Lambda$. Since $f$ is homotopic  to $f\circ\alpha\circ \beta$, by Theorem \ref{DiscreteHomotopyLifting} $f$ lifts to a (discrete) asymptotic homomorphism.
\end{proof}

\begin{remark} All statements proved in this section hold true when we replace liftings by asymptotic liftings.
\end{remark}

\section{An application: MF-algebras}

Recall that a $C^*$-algebra $A$ is {\bf MF} (or {\bf matricial field}) if it embeds into $\prod M_{k_n}/\oplus M_{k_n}$, for some $k_n\in \mathbb N$.

Equivalently, $A$ is MF if there exist maps $\phi_n: A \to M_{k_n}$, for some $k_n\in \mathbb N$, which are approximately multiplicative, approximately linear, approximately self-adjoint, and approximately injective. Reformulating it "locally", $A$ is MF if for any $F\subset\subset A$ and $\epsilon >0$ there is $k$ and a map $\phi_k: A \to M_k$ such that
\begin{multline*}\|\phi_(a)\|>\|a\|-\epsilon, \;\|\phi(ab)- \phi(a)\phi(b)\|\le \epsilon,\\
\|\phi(a+b)- \phi(a)-\phi(b)\|\le \epsilon, \;\|\phi(a^*)- \phi(a)^*\|\le \epsilon,\end{multline*} for any $a, b \in F$.

\medskip

First we will obtain a lifting characterization of MF algebras. We will use a quotient  map constructed as follows. Let $H$ be a Hilbert space and let $P_{n}$, $n \in \mathbb N$,  be an increasing sequence of projections of dimension $n$ that $\ast$-strongly converge to $1_{B(H)}$. We will identify $M_{n}$ with $P_{n} B(H) P_{n}$.
Let $\mathcal D \subset \prod_{n\in \mathbb N} M_{n}$ be the C*-algebra of all $\ast$-strongly convergent sequences of matrices. Let $q: \mathcal D \to B(H)$ be the surjection that sends each sequence  to its $\ast$-strong limit.
%\subsection{A lifting characterization of MF algebras}
Our main tool is a lifting characterization of separable RFD C*-algebras obtained by Don Hadwin.

\begin{theorem}\label{RFD} (Hadwin \cite{DonRFD}) Let $A$ be separable. TFAE:
\begin{itemize}

\item[(i)] $A$ is RFD,

\item[(ii)] every $\ast$-homomorphism from $A$ to $B(H)$ lifts to a $\ast$-homomorphism from $A$ to $\mathcal D$.
\end{itemize}
\end{theorem}

\begin{theorem}\label{characterizationMF} Let $A$ be separable. TFAE:

\begin{itemize}

\item[(i)] $A$ is MF,

\item[(ii)] every $\ast$-homomorphism from $A$ to $B(H)$ lifts to a discrete asymptotic homomorphism from $A$ to $\mathcal D$,

\item[(iii)] every $\ast$-homomorphism from $A$ to $B(H)$ asymptotically lifts to a discrete asymptotic homomorphism from $A$ to $\mathcal D$.

\item[(iv)] there exists an embedding of $A$ into $B(H)$ that asymptotically lifts to a discrete asymptotic homomorphism from $A$ to $\mathcal D$.
\end{itemize}
\end{theorem}
\begin{proof}  (i)$\Rightarrow$ (ii): By \cite[Prop.11.1.8]{BO} $A$ can be written as inductive limit $A = \varinjlim A_n$, where each $A_n$ is RFD. Let $\theta_{n, m}: A_n\to A_m$ and $\theta_{n, \infty}: A_n\to A$ be the corresponding connecting $\ast$-homomorphisms.

Let $f: A \to B(H)$ be a $\ast$-homomorphism. By Theorem \ref{RFD} $f\circ \theta_{n, \infty}$ lifts to a $\ast$-homomorphism $\psi_n: A_n\to \mathcal D$. Let $s: A\to A_1$ be any (not even continuous) section of $\theta_{1, \infty}$. For each $n\in \mathbb N$, define a map $\phi_n: A \to \mathcal D$ by
$$\phi_n(a) = \psi_n(\theta_{1, n}(s(a))),$$ $a\in A$. Then for any $a, b\in A$
$$\|\phi_n(a)\phi_n(b)-\phi_n(ab)\| = \|\psi_n(\theta_{1,n}(s(a)s(b)-s(ab)))\| \le \|\theta_{1,n}(s(a)s(b)-s(ab))\|,$$
and therefore
\begin{multline*}\limsup_n \|\phi_n(a)\phi_n(b)-\phi_n(ab)\| \le \limsup_n \|\theta_{1,n}(s(a)s(b)-s(ab))\| \\
= \|\theta_{1,\infty}(s(a)s(b)-s(ab))\| =0\end{multline*}
(we used here that $\limsup_n \|\theta_{1,n}(x)\| = \|\theta_{1,\infty}(x)\|$, $x\in A_1$, see \cite[Th. 13.1.2]{LoringBook}).
One similarly checks approximate linearity  and self-adjointness of $\phi_n$, $n\in \mathbb N$. Therefore it is a discrete asymptotic homomorphism.
Since for any $n\in \mathbb N$, $a\in A$
$$q\circ\phi_n(a) = q(\psi_n(\theta_{1, n}(s(a))) = f(\theta_{n, \infty}(\theta_{1, n}(s(a)))) = f(\theta_{1, \infty}(s(a))) = f(a),$$
we conclude that $\phi_n$ is a lift of $f$, $n\in \mathbb N$.

\medskip

(ii)$\Rightarrow$(iii) and  (iii)$\Rightarrow$(iv) are obvious.

\medskip

(iv)$\Rightarrow$(i): Let $F\subset \subset A$, $\epsilon >0$. Let $\pi: A \to B(H)$ be an embedding that asymptotically lifts to an asymptotic homomorphism $\phi_n: A \to \mathcal D$, $n\in \mathbb N$. Then there is $N_1$ such that for any $n>N_1$
\begin{equation}\label{characterizationMF1}\|\phi_n(ab) - \phi_n(a)\phi_n(b)\|\le \epsilon,\end{equation} $a, b\in F$, and similarly for sums and adjoint element. Since for each $a\in A$, $q\circ\phi_n(a)\to \pi(a)$,  there is $N_2> N_1$ such that for any $n>N_2$
\begin{equation}\label{norm} \|\phi_n(a)\|>\|a\|-\epsilon,
\end{equation}
$a\in F$. Fix $m> N_2$. Write $\phi_m = \left(\phi_{m, k}\right)_{k\in \mathbb N},$ $\phi_{m, k}: A \to M_k$. It follows from (\ref{norm}) that there is $k$ such that
\begin{equation}\label{characterizationMF3}\|\phi_{m, k}(a)\|>\|a\|-2\epsilon,\end{equation} $a\in F$. It follows from (\ref{characterizationMF1}) that
\begin{equation}\label{characterizationMF4}\|\phi_{m,k}(ab) - \phi_{m, k}(a)\phi_{m, k}(b)\|\le \epsilon.\end{equation}
By (\ref{characterizationMF3}) and (\ref{characterizationMF4}), $A$ is MF.
\end{proof}

Now we are ready to prove the main result of this subsection.

\begin{theorem}\label{MFhominv} Let $A$ and $B$ be separable.  If $A$ is homotopically dominated by $B$, and $B$ is MF, then $A$ is also MF.  In particular, MF-property is homotopy invariant.
\end{theorem}
\begin{proof} It follows from the equivalence (i)$\Leftrightarrow$(ii) in Theorem \ref{characterizationMF} and Corollary \ref{HomInvGen}.
\end{proof}

\begin{remark}  In \cite{MT2}  there is an example of a separable C*-algebra $A$ that admits a $\ast$-homomorphism $f$ to a quotient, $B/I$, such that for any $\ast$-homomorphism $g:A \to B/I $, $f\oplus g$ does not lift to an asymptotic homomorphism. However in that example $f$ admitted a lift to a discrete asymptotic homomorphism. Now we can construct many examples of $\ast$-homomorphisms that even after adding any other $\ast$-homomorphism admit no lift  even to a discrete asymptotic  homomorphism. Indeed take any non-MF C*-algebra, e.g. Toeplitz algebra $\mathcal T$. Let $f: \mathcal T \to B(H)$ be any embedding. Then for any $g: \mathcal T \to B(H)$, $f\oplus g$ is also an embedding. By Theorem \ref{characterizationMF}, $f\oplus g$ does not lift to a discrete asymptotic homomorphism.
\end{remark}

\section{A cp version of homotopy lifting}

\begin{lemma}\label{cpVersionCylindersLemma1} Let $\psi: B\to A$ be a $\ast$-homomorphism and suppose $A$ is unital. For any $\ast$-homomorphism $F: Z_{\psi}\to D$, its restrictions $f= F|_{CA}$ and $g= F|_B$ satisfy the relation
\begin{equation}\label{restrictions} f(t\otimes 1_A)g(b) = f (t\otimes \psi(b)),\end{equation} for any $b\in B$.

The other way around, for any $\ast$-homomorphisms $f: CA \to D$ and $g: B\to D$  satisfying (\ref{restrictions}), there exists unique $\ast$-homomorphism $F_{f, g}: Z_{\psi}\to D$  whose restrictions to $CA$ and $B$ equal to $f$ and $g$ respectively.
%such that $F_{f, g}|_{CA} = f, F_{f, g}|_B = g$.
\end{lemma}
\begin{proof} The first statement is obvious. For the second one, we will use the presentation of $Z_{\psi}$  from Theorem \ref{presentation} and notation used there. We define $$F(\mathbf x') = g(\mathbf x), $$
$$F(\mathbf y') = f( \mathbf y''),$$
$$F(\mathbf z') = f( \mathbf z''),$$
$$F(h) = f(k).$$
             Then
$$F(h)F(\mathbf x') = f( k)g(\mathbf x) = f(t\otimes 1_A) g(\mathbf x) = f(t\otimes \psi(\mathbf x)) = f(\mathbf z'') = F(\mathbf z').$$
All the other relations of $Z_{\psi}$ are clearly satisfied.
\end{proof}

%\begin{lemma} Let $\phi: A \to D$ be an order zero map and $g: B\to D$ a $\ast$-homomorphism satisfying  $$\phi(1_A)g(b) = \phi(\psi(b)), $$ for any $b\in B$. Then $\rho_{\phi}$ and $g$ satisfy (\ref{restrictions}). \end{lemma}

%\begin{lemma}\label{cpVersionCylindersLemma2} Let $A$ be unital, $\psi: B\to A$ a $\ast$-homomorphism. Let $\phi: A \to D$ be an order zero map and $g: B\to D$ a $\ast$-homomorphism satisfying \begin{equation}\label{restrictions2}\phi(1_A)g(b) = \phi(\psi(b)), \end{equation} for any $b\in B$. Then there is a unique $\ast$-homomorphism $F_{\phi, g}: Z_{\psi} \to D$ such that $\phi = F_{\phi, g}|_{CA}\circ \delta$ and $g = F_{\phi, g}|_B$.\end{lemma}
%\begin{proof} Since $\phi = \rho_{\phi}\circ \delta$,  $\rho_{\phi}$ and $g$ satisfy (\ref{restrictions}). By Lemma \ref{cpVersionCylindersLemma1} there is unique $\ast$-homomorphism $F: Z_{\psi} \to D$ whose restrictions on $CA$ and $B$ equal to $\rho_{\phi}$ and $g$ respectively. We will denote this $\ast$-homomorphism by $F_{\phi, g}$.  Since  $\rho_{\phi}$ is the only $\ast$-homomorphism satisfying $\phi = \rho_{\phi}\circ \delta$, a $\ast$-homomorphism with the needed property is unique. \end{proof}

Below let $\delta: A \to CA$ be defined by $\delta(a) = t\otimes a$, $a\in A$.

\begin{lemma}\label{cpVersionCylindersLemma3} Let $A$ be unital, $\psi: B\to A$ a $\ast$-homomorphism. Let $f: CA \to D/I$  and $g: B\to D/I$ be  $\ast$-homomorphisms satisfying
$$f(t\otimes 1_A)g(b) = f (t\otimes \psi(b)), $$
for any $b\in B$. Suppose both $f\circ\delta$ and $g$ have cp lifts. Then $F_{f, g}: Z_{\psi}\to D/I$ has a  cp lift.
\end{lemma}
\begin{proof}  Let $F$ be a finite subset of the unit ball of $Z_{\psi}$ and  $\epsilon >0$.
There exists $\delta>0$ such that
\begin{equation}\label{LiftingDuality1}\|\eta(x) - \eta(x')\|< \epsilon, \;\text{whenever}\; |x-x'|\le \delta,\end{equation}
for any $(\eta, b)\in F$.
There exists $\delta'$ such that
$$\frac{t_1}{t_2}< 1 + \epsilon,\; \text{whenever}\; t_1, t_2 \in [\delta/2, 1] \; \text{and}\; |t_1-t_2|< \delta'$$
(for instance one can take $\delta' = \epsilon\delta/2$).
Let $I_0, \ldots, I_N$ be a cover of $[0, 1]$ by intervals, such that $I_0= [0, \delta]$, and for each $i\ge 1$, $I_i$ has length not larger than $\delta'$ and $I_i\bigcap [0, \frac{\delta}{2}] = \emptyset.$ Let $\{u_i\}_{i=0}^N$ be the corresponding partition of unity. We have $u_0(0) = 1$.
Let $t_0=0$ and choose arbitrary $t_i \in I_i$, for $i\ge 1$.
Let
$$W = \{(a_0, \ldots, a_N, b)\;|\; a_i\in A, b\in B, \psi(b) = a_0\} \subset A^{\oplus (N+1)} \oplus B.$$
 We define  a $\ast$-homomorphism $\alpha: Z_{\psi} \to W$ by
$$\alpha((\eta, b)) = (\eta(t_0), \ldots, \eta(t_N), b), $$ $(\eta, b)\in Z_{\psi}$.  We define a cp map $\beta: W \to Z_{\psi}$ by
$$\beta\left((a_0, \ldots, a_N, b)\right) = (\sum_{i=1}^N \frac{a_i}{t_i}\otimes tu_i, 0) + (a_0\otimes u_0, b).$$

{\it Claim 1:} For any $(\eta, b)\in F$, $$\|\beta\circ\alpha ((\eta, b)) - (\eta, b)\|< 3\epsilon.$$

{\it Proof of Claim 1:} For any $(\eta, b)\in Z_{\psi}$,
$$\beta\circ\alpha((\eta, b)) = (\sum_{i=1}^N \frac{\eta(t_i)}{t_i}\otimes u_it, 0) + (\eta(0)\otimes u_0, b).$$
Since $Z_{\psi}$ is spanned by $CA$ and $B$, WLOG we can assume that $F \subset CA \bigcap B$.
For any $(\eta, 0) \in F\bigcap CA$ we have

\begin{multline*} \|\beta\circ\alpha((\eta, 0)) - (\eta, 0)\| = \|\sum_{i=1}^N \frac{\eta(t_i)}{t_i}\otimes u_it - \eta\|\\
= \sup_{t\in [0, 1]}\|\sum_{i=1}^N \frac{\eta(t_i)}{t_i}u_i(t)t - \eta(t)\sum_{i=1}^N u_i(t) - \eta(t) u_0(t)\|\\
\le \sup_{t\in [0, 1]}\|\sum_{i=1}^N \left(\frac{\eta(t_i)}{t_i}t - \eta(t)\right)u_i(t) - \eta(t)u_0(t)\|\\
\le  \sup_{t\in [0, 1]}\sum_{i=1}^N |\left(\|\eta(t_i)(\frac{t}{t_i} - 1)\| + \|\eta(t_i)-\eta(t)\|\right)u_i(t)| + \sup_{t\in [0, \delta]} \|\eta(t)\|\\ \le \sup_{t\in [0, 1]}\sum_{i=1}^N 2\epsilon u_i(t) + \epsilon \le 3\epsilon.
\end{multline*}

For any $(\psi(b)\otimes 1, b) \in F\bigcap B$ we have

\begin{multline*} \|\beta\circ\alpha((\psi(b)\otimes 1, b)) - (\psi(b)\otimes 1, b)\| \\
=\|(\sum_{i=1}^N \frac{\psi(b)}{t_i} \otimes u_it, 0) + (\psi(b)\otimes u_0, b) - (\psi(b)\otimes 1, b) \|\\
= \|(\sum_{i=1}^N \frac{\psi(b)}{t_i} \otimes u_it, 0) - \sum_{i=1}^N (\psi(b)\otimes u_i, b)\|
= (\sum_{i=1}^N (\frac{\psi(b)}{t_i} \otimes u_it - \psi(b)\otimes u_i), -b)\\ =
\|b\|\| \sum_{i=1}^N (\frac{\psi(b)}{t_i} \otimes u_it - \psi(b)\otimes u_i)\|\\
\le \sup_{t\in [0, 1]}\sum_{i: t\in I_i, i\ge 1}\|(\frac{\psi(b)}{t_i}t - \psi(b))u_i(t)\| \\
\le \sup_{t\in [0, 1]}\sum_{i: t\in I_i, i\ge 1} \|\psi(b)\||\frac{t}{t_i}-1|u_i(t) \le \sup_{t\in [0, 1]}\sum_{i: t\in I_i, i\ge 1}\epsilon u_i(t) \le \epsilon.\end{multline*}
Claim 1 is proved.

\medskip

{\it Claim 2:} $F_{f, g}\circ \beta$, and therefore $F_{f, g}\circ \beta\circ \alpha$,   lifts to a cp map.

{\it Proof of Claim 2:}  \begin{multline}\label{ccpDualityCylinder1}F_{f, g}\circ \beta((a_0, \ldots, a_N, b)) = \sum_{i=1}^N F_{f, g}((\frac{a_i}{t_i}\otimes u_it, 0)) +
F_{f, g}((a_0\otimes u_0, b)).\end{multline}
 Let $y_i: = F_{f, g}(((\frac{1_A}{t_i}\otimes u_i)^{1/2}, 0))$.  We have
\begin{multline}\label{ccpDualityCylinder2}F_{f, g}((\frac{a_i}{t_i}\otimes u_it, 0))  = F_{f, g}\left(((\frac{1_A}{t_i}\otimes u_i)^{1/2} (t\otimes a_i)(\frac{1_A}{t_i}\otimes u_i)^{1/2}, 0)\right)\\ = F_{f, g}(((\frac{1_A}{t_i}\otimes u_i)^{1/2}, 0)) F_{f, g}(t\otimes a_i, 0) F_{f, g}(((\frac{1_A}{t_i}\otimes u_i)^{1/2}, 0)) = y_i f\circ \delta(a_i) y_i.\end{multline}
Since $\phi\circ\delta$ lifts to a cp map, so does the map $y_i \phi\circ\delta y_i$.
Let $\{e_{\lambda}\}$ be an approximate unit in $B$. Let $z_{\lambda}:= F_{f, g}\left((\psi(e_{\lambda})^{1/2}\otimes u_0^{1/2}, e_{\lambda})\right)$. We have
$$(a_0\otimes u_0, b) = (\psi(b)\otimes u_0, b) = \lim_{\lambda}(\psi(e_{\lambda})^{1/2}\otimes u_0^{1/2}, e_{\lambda})(\psi(b)\otimes 1, b)(\psi(e_{\lambda})^{1/2}\otimes u_0^{1/2}, e_{\lambda})$$ and therefore
\begin{multline}\label{ccpDualityCylinder3}
F_{f, g}((a_0\otimes u_0, b)) = \lim_{\lambda}  z_{\lambda} F_{f, g}\left((\psi(b)\otimes 1, b)\right) z_{\lambda}
= \lim_{\lambda}  z_{\lambda} g(b) z_{\lambda}.\end{multline}
Since $g$ lifts to a cp map, so does the map $z_{\lambda} gz_{\lambda}$, and by Arveson's theorem, so does $\lim_{\lambda}z_{\lambda} gz_{\lambda}$.
Therefore, by  (\ref{ccpDualityCylinder1}), (\ref{ccpDualityCylinder2}), (\ref{ccpDualityCylinder3}), $F_{f, g}\circ\beta$ is sum of liftable cp maps and hence is liftable.
 Claim 2 is proved.

\medskip
Let $\{a_1, a_2, \ldots\}$ be a dense subset of the unit ball of $A$.
Taking $\epsilon = 1/n$ and $F= \{a_1, \ldots, a_n\}$ , by Claim 1 we obtain that $F_{f, g}$ is pointwise limit of maps $F_{f, g}\circ \beta\circ\alpha$ which lift to cp  maps by Claim 2. By Arveson's theorem, $F_{f, g}$ lifts to a cp map.
\end{proof}

\begin{theorem}\label{ccpcylinder}  Let $F:Z_{\psi}\to D/I$ be a $\ast$-homomorphism. Suppose $F|_B$ lifts to a ccp (discrete) asymptotic homomorphism and  $F|_{CA}$ lifts to a ccp map.   Then $F$ lifts to a ccp (discrete) asymptotic homomorphism.
\end{theorem}
\begin{proof} We can assume that $F$ is surjective.
%Similar to the proof of Theorem \ref{ccp}.
If $A$ is non-unital, let $i: A \to A^+$ be the inclusion map. Since $Z_{\psi}$ is an essential ideal in $Z_{i\circ\psi}$, we have $Z_{\psi} \subset Z_{i\circ \psi} \subset M(Z_{\psi})$. By the NC Tietze
Extension Theorem $F$ extends to a $\ast$-homomorphism $F': M(Z_{\psi}) \to M(D/I)$.
Then $\tilde F = F'|_{Z_{i\circ \psi}}:Z_{i\circ\psi}\to M(D)/I$ is an
 extension of $F$.
Since the image of $B$ in both $Z_{\psi}$ and $Z_{i\circ\psi}$ is the same, $\tilde F|_B$ lifts to a cpc asymptotic homomorphism $\Phi_{\lambda}$, $\lambda\in \Lambda$.
  By  \cite[Lemma 18]{sectionsCones}, $\tilde F|_{C(A^+)} $ lifts to a ccp map $\Phi$.

If $A$ is unital, we have $\tilde F = F$. Now we proceed with both unital and non-unital cases simultaneously.

Let $\{i_{\lambda}\}$ be an approximate unit in $I$ that is quasicentral for $D$ (and is a continuous path in the continuous case, see Lemma \ref{continuousqau}). Let
$$\phi_{\lambda} = (1-i_{\lambda})^{1/2} \Phi(1-i_{\lambda})^{1/2}: C(A^+)\to M(D),$$ $\lambda\in \Lambda$. Let  $ a_1, a_2 \ldots$ be a dense subset of $A$. Let $x_i = \delta(a_i)$, $h = \delta(1_{A^+})$.
We write $C(A^+)$ as the universal C*-algebra with generators $h$ and $x_i$, $i\in \mathbb N$, and homogeneous relations.   By Lemma \ref{AsHom1}  $\phi_{\lambda}(h)$ and $\phi_{\lambda}(x_i)$'s approximately satisfy the relations and therefore define a $\ast$-homomorphism $\phi: C(A^+)\to C_b(\Lambda, M(D))/C_0(\Lambda, I)$ (note that in the discrete case $C_b(\Lambda, M(D))/C_0(\Lambda, I) = \prod M(D)/\oplus I$). We note that $(\phi_{\lambda})_{\lambda\in \Lambda}$ need not be a lift of $\phi$. It only lifts $\phi$ on the linear span of the generators.  However $(\phi_{\lambda})_{\lambda\in \Lambda}\circ \delta$ is a lift of $\phi\circ \delta$ because $(\phi_{\lambda})_{\lambda\in \Lambda}\circ \delta(a_i) = (\phi_{\lambda}(x_i))_{\lambda\in \Lambda}$ is a lift of $\phi(x_i) = \phi\circ \delta(a_i)$.

The asymptotic homomorphism $\Phi_{\lambda}, \lambda\in \Lambda$, gives rise to a $\ast$-homomorphism  $g: B \to C_b(\Lambda, M(D))/C_0(\Lambda, I)$. By Lemma \ref{AsHom1} $\phi_{\lambda}$ and $\Phi_{\lambda}$ approximately satisfy (\ref{restrictions}). Therefore $\phi$ and $g$ satisfy (\ref{restrictions}) precisely. By Lemma \ref{cpVersionCylindersLemma1} they define a $\ast$-homomorphism $F_{\phi, g}: Z_{\psi\circ i}\to C_b(\Lambda, M(D))/C_0(\Lambda, I)$. By Lemma \ref{cpVersionCylindersLemma3} $F_{\phi, g}$ lifts to a cp map
$\left( \tilde  F_{\lambda}\right)_{\lambda\in \Lambda}: Z_{\psi\circ i}\to C_b(\Lambda, M(D))$. We note that if $A$ is unital, then $\left(\tilde F_{\lambda}\right)_{\lambda\in \Lambda}$ lands in $C_b(\Lambda, D)$. Since for any $x, y\in Z_{\psi\circ i}$
$$\tilde F_{\lambda}(xy) - \tilde F_{\lambda}(x)\tilde F_{\lambda}(y)\in \oplus I,$$ $\tilde F_{\lambda}, \lambda\in \Lambda$, is an asymptotic homomorphism and for each $\lambda\in \Lambda$,  $q\circ \tilde F_{\lambda}$ is a $\ast$-homomorphism.   It remains to prove that $q\circ \tilde F_{\lambda}|_{Z_{\psi}}= F$, for each $\lambda$.
 Since $(\phi_{\lambda}(x_i))_{\lambda\in \Lambda}$ and $\left(\tilde F_{\lambda}|_{C(A^+)}(x_i)\right)_{\lambda\in \Lambda}$ are both lifts of $\phi(x_i)$, we have
 $$(\phi_{\lambda}(x_i))_{\lambda\in \Lambda} - \left(\tilde F_{\lambda}|_{C(A^+)}(x_i)\right)_{\lambda\in \Lambda}\in C_0(\Lambda, I)$$ and therefore
 $$\phi_{\lambda}(x_i) - \tilde F_{\lambda}|_{C(A^+)}(x_i) \in I, $$ for each $\lambda$. Hence $$q\circ \tilde F_{\lambda}|_{CA}(x_i) = q\circ \phi_{\lambda}(x_i) = F(x_i).$$
 Thus \begin{equation}\label{equalOnCA}q\circ\tilde F_{\lambda}|_{CA} = F|_{CA}.\end{equation}
  Similarly, since $(\tilde F_{\lambda}|_B)_{\lambda\in \Lambda} - (F_{\lambda})_{\lambda\in \Lambda}\in C_0(\Lambda, B)$, we conclude that
  \begin{equation}\label{equalOnB}q\circ \tilde F_{\lambda}|_B = F|_B.\end{equation} Since $CA$ and $B$ generate $Z_{\psi}$, by (\ref{equalOnCA}) and (\ref{equalOnB}) $q\circ \tilde F_{\lambda}|_{Z_{\psi}} = F$, for each $\lambda$.

\end{proof}

Now we obtain a cp version of homotopy lifting. We note that it is not as general as Theorem \ref{DiscreteHomotopyLifting}.
 %The reason is that the theorem above  requires $f|_{CA}$ to have a cp lift, and it is not clear how to arrange it in general situation.

\begin{theorem}\label{cpHomotopyLiftingCor1} Let $\phi: B \to D/I$ be a $\ast$-homomorphism that has a cp lift. Let $\psi', \psi'': C\to B$ be homotopic $\ast$-homomorphisms and suppose $\phi\circ \psi''$ lifts  to a  cp (discrete) asymptotic homomorphism.
Then $\phi\circ \psi'$ lifts  to a cp (discrete) asymptotic homomorphism.
\end{theorem}
\begin{proof} Since $\psi'$ is homotopic to $\psi''$, by Lemma \ref{factorization} it factorizes through $Z_{\psi''}$

$$ \begin{tikzcd}  B \arrow{rr}{\psi'}\arrow{rd}{\alpha}& & A \arrow{r}{\phi} & D/I \\
& Z_{\psi''} \arrow{ru}{\beta}& &  \end{tikzcd}$$

\noindent such that $\beta|_{B} = \psi''$. Therefore $\left(\phi\circ \beta\right)|_{B} = \phi \circ \psi''$ lifts to a cp (discrete) asymptotic homomorphism.
Since $\phi$ has a cp lift, so does $\phi\circ \beta$. By Theorem \ref{ccpcylinder}, $\phi\circ \beta$ lifts to a cp (discrete) asymptotic homomorphism. Then $\phi\circ\phi' = \phi\circ\beta\circ\alpha$ also lifts to a cp (discrete) asymptotic homomorphism.
\end{proof}

\begin{theorem}\label{cpHomotopyLiftingCor2} Suppose $A$ is homotopically dominated by $B$ and each $\ast$-homomorphism from $B$ to $D/I$ that has a cp lift  lifts to a cp (discrete) asymptotic homomorphism. Then each $\ast$-homomorphism from $A$ to $D/I$ that has a cp lift  lifts to a cp (discrete) asymptotic homomorphism.
\end{theorem}
\begin{proof} Since $A$ is homotopically dominated by $B$, there are $\ast$-homomorphisms $\alpha: B \to A$ and $\beta: A\to B$ such that
\begin{equation}\label{cpHomotopyLiftingCor21} \alpha
\circ \beta \sim id_A.\end{equation}
Let $\phi: A \to D/I$ be a $\ast$-homomorphism that has a cp lift.
 Then $\phi\circ \alpha$ also has a cp lift. Then, by assumption, $\phi\circ \alpha$ lifts to a cp (discrete) asymptotic homomorphism. Hence so does $\phi\circ\alpha\circ\beta$. By (\ref{cpHomotopyLiftingCor21}) and Theorem \ref{cpHomotopyLiftingCor1}, $\phi = \phi \circ id_A$ lifts to a cp (discrete) asymptotic homomorphism.
\end{proof}

\section{An application: Traces and  homotopy invariance}

In \cite{Neagu} R. Neagu raised a question of whether the property that all amenable traces are quasidiagonal is homotopy invariant. He proved
that if $ A$ is  a separable exact C*-algebra with a faithful
amenable trace,  $A$ is homotopy dominated by a C*-algebra $B$ and all amenable traces on $B$ are quasidiagonal, then all
amenable traces on $A$ are quasidiagonal.  Brown-Carrion-White's result for cones, which generalizes to the class of all contractible C*-algebras (\cite[Prop.1.4]{Neagu}), is also a particular case of the above question with $B=0$.

We will prove a few more partial results on this topic. In particular it will be proved that if either $A$ or $B$ is exact, then the question above has a positive answer.

\medskip

First, we obtain a  characterization of MF- and quasidiagonal traces in terms of liftings. In \cite[Prop. 1.3]{Schafhauser}  there are lifting reformulations of the notions of amenable and quasidiagonal traces, where lifting means cp-lifting.   In particular, for an MF-trace $\tau$,  Schafhauser's reformulation  would sound as $\tau$ being  a trace of some homomorphism to $Q_{\omega}$, so there is no lifting there. We will need a different reformulation, in terms of asymptotic liftings.

\begin{proposition}\label{propositionMFtrace} Let $\tau$ be a trace on a C*-algebra $A$. TFAE:
\begin{itemize}

\item[(i)] $\tau$ is an MF (quasidiagonal, respectively) trace,

\item[(ii)]  there exists a $\ast$-homomorphism $f: A \to \prod M_{m_n}/\oplus_{2, \omega} M_{m_n}$ such that
$\tau = tr f$  and $f$ lifts to a (cp) discrete asymptotic homomorphism from $A$ to $\prod M_n$,

\item[(iii)]  there exists a discrete asymptotic homomorphism $f^{k}: A \to \prod M_{m_n}/\oplus_{2, \omega} M_{m_n}$, $k\in \mathbb N$,  such that
$\tau(a) = \lim_{k\to \infty} tr f^{k}(a)$, for any $a\in A$,   and $f^{k}$ lifts asymptotically to a (cp) discrete asymptotic homomorphism from $A$ to $\prod M_{m_n}$.
\end{itemize}
\end{proposition}
\begin{proof} (i)$\Rightarrow$ (ii): Since $\tau$ is an MF (quasidiagonal, respectively) trace, there exists an approximately multiplicative family of (cp) maps $\phi_n: A \to M_{m_n}$, $n\in \mathbb N$, such that
$$\tau(a) = \lim_{n\to \infty} tr_{m_n} \phi_n(a),$$ for any $a\in A$. Let
$$\psi^k = (0, \ldots, 0, \phi_k, \phi_{k+1}, \ldots): A \to \prod M_{m_n}.$$
Then
$$\|\psi^k(ab)-\psi^k(a)\psi^k(b)\|=\sup_{n>k} \|\phi_n(ab)-\phi_n(a)\phi_n(b)\|=0,$$
for any $a, b\in A$, and asymptotic linearity is similar, so we got a (cp) asymptotic homomorphism.
Let $f^k = q\circ \psi^k.$ It follows from construction of $\psi^k$ that all $f^k$ are the same, and we denote
$f= f^k$. Then $f$ is a $\ast$-homomorphism and
$$tr f(a) = \lim_{n\to\omega} tr_{m_n}\psi^k_n(a) = \lim_{k<n\to \omega} tr_{m_n} \phi_n(a) = \tau(a),$$
$a\in A$.

\medskip

(ii)$\Rightarrow$ (iii) is clear.

\medskip

(iii)$\Rightarrow$ (i):  WLOG we can assume $m_n=n$.  Write
$f^k = q\left(\left( f^k_n\right)_{n\in \mathbb N}\right)$. Let $F\subset \subset A$, $\epsilon >0$. There exists $k_0$ such that the following 3 conditions hold.

\medskip

1) $|\tau(a) - tr f^{k_0}(a)|< \epsilon, $ for any $a\in F$.

In particular this condition implies that there is $E\subset \omega$ such that  for any $n\in E$
\begin{equation}\label{propositionMFtrace1}  |\tau(a) - tr_nf^{k_0}_n(a)|<\epsilon,\end{equation}
for any $a\in F$.

\medskip

2) $\|\psi^{k_0}(ab)- \psi^{k_0}(a)\psi^{k_0}(b)\|<\epsilon,$ for any $a, b\in F$.

In particular this implies that for each $n$
\begin{equation}\label{propositionMFtrace2} \|\psi^{k_0}_n(ab) - \psi^{k_0}_n(a)\psi^{k_0}_n(b)\|<\epsilon, \end{equation}
for any $a, b\in F$.

\medskip

3) $\|q\circ \psi^{k_0}(a) - f^{k_0}(a)\|<\epsilon, $ for any $a\in F$.

Since  $\|q\circ \psi^{k_0}(a) - f^{k_0}(a)\| = \limsup_n \|\psi^{k_0}_n(a)- f^{k_0}_n(a)\|$, this condition implies that there is $n_0\in E$
such that
\begin{equation}\label{propositionMFtrace3} \|\psi_{n_0}^{k_0}(a)- f_{n_0}^{k_0}(a)\|< \epsilon, \end{equation}
for any $a\in F$.

\medskip

Then by (\ref{propositionMFtrace1}) and (\ref{propositionMFtrace3})
$$|\tau(a)- tr_{n_0} \psi^{k_0}_{n_0}(a)|\le |\tau(a) - tr_{n_0} f^{k_0}_{n_0}(a)| + |tr_{n_0} f^{k_0}_{n_0}(a) - tr \psi^{k_0}_{n_0}(a)| < 2\epsilon,$$
for any $a\in F$, and by (\ref{propositionMFtrace2}) $\psi^{k_0}_{n_0}$ is $\epsilon$-multiplicative on $F$. This means $\tau$ is an MF (quasidiagonal) trace.
\end{proof}

 \medskip Proving statements that say that if $A$  is homotopically dominated by $B$ and $B$ has some property, then  $A$ also has the same property, one can put additional assumptions either on $A$ (as in Neagu's theorem)  or on $B$ (as in Brown-Carrion-White theorem).

 \subsection{Additional assumptions on $A$}
 In this subsection we will prove that in Neagu's theorem the requirement that $A$ has  a faithful trace can be removed.

 The next lemma is well-known.

  \begin{lemma}\label{bicommutant} Let $A$ be a C*-algebra, $\mathcal M$ be a von Neumann algebra with a normal faithful tracial state $\sigma$ and $f: A \to \mathcal M$ a $\ast$-homomorphism. Then there exists a $\ast$-homomorphism $\tilde f: \pi_{\sigma\circ f}(A)'' \to \mathcal M$ such that $$f= \tilde f \circ \pi_{\sigma\circ f}.$$ $\sigma \circ \tilde f$ is a normal faithful tracial state on $\pi_{\sigma\circ f}(A)''$.
 \end{lemma}

 \begin{theorem}\label{Don}(Hadwin \cite{HadwinApprEquiv}) Suppose $A = W^*(x_1, \ldots, x_s)$ is a hyperfinite von Neumann
algebra with a faithful normal tracial state $\rho$. For every $\epsilon >0$ there is a $\delta>0$
and an $N\in \mathbb N$ such that, for every unital C*-algebra $B$ with a factor tracial state $\tau$
and $a_1, \ldots, a_s, b_1, \ldots, b_s\in  B$, if, for every $\ast$-monomial $m(t_1, \ldots, t_s)$ with degree at
most $N$,
$$|\tau (m (a_1, \ldots, a_s)) - \rho (m (x_1, \ldots,  x_s))| < \delta,$$
$$|\tau (m (b_1, \ldots,  b_s)) - \rho (m (x_1, \ldots,  x_s))| < \delta,$$
then there is a unitary element $u \in B$ such that
$$\sum_{k=1}^s \|ua_ku^* - b_k\|_2 < \epsilon.$$
 \end{theorem}

 Recall that $\prod M_n/\oplus_{2, \omega} M_n$ is a $II_1$-factor with a normal faithful tracial state $tr$ defined by
 $tr\; q((T_n))_{n\in \mathbb N}) = \lim_{n\to \omega } tr_n T_n.$

 \begin{corollary}\label{unitaryEquivalence} Let $A$ be a separable C*-algebra, $f, g: \mathcal A \to \prod M_n/\oplus_{2, \omega} M_n$   $\ast$-homomorphisms such that
 $tr \; f = tr \; g$ and $\pi_{tr f}(A)''$ is hyperfinite. Then  there is a unitary $u\in\prod M_n/\oplus_{2, \omega} M_n$ such that $f= u^*gu$.
 \end{corollary}
 \begin{proof} By Lemma \ref{bicommutant} it is sufficient to prove that $\tilde f$ and $\tilde g$ are unitarily equivalent. Let $a_1, a_2, \ldots$ be a dense subset of $A$. Let $\left(b^{(i)}_k\right)_{k\in \mathbb N}$ be a preimage of $f(a_i)$ in $\prod M_k$, and $\left(c^{(i)}_k\right)_{k\in \mathbb N}$ be a preimage of $g(a_i)$ in $\prod M_k$.  For any monomial $m$,  any $i$  and $s$ we have

 \begin{equation}\label{unitaryEquivalence1} tr \tilde f (m(\pi_{tr f}(a_1), \ldots, \pi_{tr f}(a_s))) = tr\; m(f(a_1), \ldots, f(a_s)) = \lim_{\omega} tr_k m(b_k^{(1)}, \ldots, b_k^{(s)}),
 \end{equation}
 \begin{equation}\label{unitaryEquivalence2} tr \tilde g (m(\pi_{tr f}(a_1), \ldots, \pi_{tr f}(a_s))) = tr\; m(g(a_1), \ldots, g(a_s)) = \lim_{\omega} tr_k m(c_k^{(1)}, \ldots, c_k^{(s)}).
 \end{equation}

 Let $n\in \mathbb N$. The von Neumann algebra $\mathcal A_0 = W^*(\pi_{tr f}(a_1), \ldots, \pi_{tr f}(a_n))$ is a von Neumann subalgebra of $\pi_{tr f}(A)''$ and hence by Connes's theorem is hyperfinite.   For $\mathcal A_0$ and $\epsilon = \frac{1}{n}$ let $\delta$ and $N$ be as in Theorem \ref{Don}.
 By (\ref{unitaryEquivalence1}) and (\ref{unitaryEquivalence2}) there exists $E_n\in \omega$ such that for any $k\in E_n$, and any  monomial $m$ of degree less than $N$
 $$|tr \tilde f \left(m((\pi_{tr f}(a_1), \ldots, \pi_{tr f}(a_n))\right) - tr_k m(b_k^{(1)}, \ldots, b_k^{(n)})|<\delta,$$
 $$|tr \tilde g \left(m((\pi_{tr f}(a_1), \ldots, \pi_{tr f}(a_n))\right) - tr_k m(c_k^{(1)}, \ldots, c_k^{(n)})|<\delta.$$
 Then, since $tr \tilde f = tr \tilde g$, by Theorem \ref{Don}, for any $k\in E_n$   there is a unitary $u_{k, n}\in M_k$ such that
  $$\|b_k^{(i)} - u_{n, k}^*c_k^{(i)}u_{n, k}\|_2 < 1/n,$$  $i = 1, \ldots, n$.

We have $$E_1\supset E_2 \supset \ldots$$  Since $\omega$ is a non-trivial ultrafilter, there is a decreasing sequence $F_n\in \omega$, $n\in \mathbb N$, with $\bigcap F_n = \emptyset$. Replacing $E_n$ with $E_n\bigcap F_n$ we can assume $$\bigcap E_n = \emptyset.$$
Then for any $k$ there exists unique $n= n(k)$ such that $k\in E_n\setminus E_{n+1}$. Let $u_k:= u_{n(k), k}.$ Then
$$\|b_k^{(i)} - u_k^*c_k^{(i)}u_k\|_2\to_{k\to \omega} 0.$$ Let $u = q\left(\left(u_k\right)_{k\in \mathbb N}\right).$ Then
$$f(a_i) = u^*g(a_i)u, $$ for each $i$.

 \end{proof}

 \begin{corollary}\label{exact} Let $A$ be a separable exact C*-algebra. Let $f, g: A \to \prod M_n/\oplus_{2, \omega} M_n$ be  $\ast$-homomorphisms such that $tr \; f$ is an amenable trace  on $A$ and
 $tr \; f = tr \; g$. Then $f = u^*gu$, for some unitary $u\in \prod M_n/\oplus_{2, \omega} M_n$.
 \end{corollary}
 \begin{proof}   Since $A$ is exact and $tr f$  is amenable, by \cite[Cor. 4.3.6]{BrownTraces} $\pi_{tr f}(A)''$  is hyperfinite.
 The statement follows now from the previous corollary.
 \end{proof}

 \begin{theorem}\label{qdTracesExact}  Suppose $A$ is a separable exact C*-algebra, $A$ is homotopy dominated by a C*-algebra $B$, and all amenable traces on $B$ are quasidiagonal. Then all amenable traces on $A$ are quasidiagonal.
 \end{theorem}
 \begin{proof} By assumption there exist $\ast$-homomorphisms $\alpha: B\to A$ and $\beta: A \to B$ such that $\alpha\circ \beta$ is homotopic to $id_A$. Let $\tau$ be an amenable trace on $A$.  This reformulates as  $\tau = tr \; f$, where $f: A \to \prod M_n/\oplus_{2, \omega} M_n$ is a $\ast$-homomorphism that has a cp lift to $\prod M_n$.
 Since $\tau\circ \alpha$ is an amenable trace on $B$, it is quasidiagonal. By Proposition \ref{propositionMFtrace} there exists a $\ast$-homomorphism $g: B \to \prod M_n/\oplus_{2, \omega} M_n$ such that
 $\tau\circ \alpha = tr \; g$ and $g$ lifts to a cp discrete asymptotic homomorphism. We have
 $$tr\; g\circ \beta = \tau \circ \alpha \circ \beta = tr \;f\circ \alpha \circ \beta .$$
 Since $tr \;f\circ \alpha \circ \beta $ is amenable and  $A$ is exact, by Corollary \ref{exact}  there exists a unitary $u\in \prod M_n/\oplus_{2, \omega} M_n$ such that
 $$u^*(g\circ \beta)u  = f\circ \alpha\circ \beta.$$ Since unitaries in $\prod M_n/\oplus_{2, \omega} M_n$ lift to unitaries in $\prod M_n$ and $g\circ \beta$ lifts to a cp discrete asymptotic homomorphism since $g$ does, so does $f\circ \alpha\circ \beta$.
 Since $id_A$ is homotopic to $\alpha \circ \beta$, and $f$ has a cp lift, by Theorem \ref{cpHomotopyLiftingCor1} $f$ lifts to a cp discrete asymptotic homomorphism.  By Proposition \ref{propositionMFtrace} again,  $\tau$ is  quasidiagonal.
\end{proof}

\medskip

\subsection{Additional assumptions on $B$}
Now we will put assumptions on $B$ instead of putting them on $A$.

\begin{theorem}\label{HomLiftTraces} If all $\ast$-homomorphisms from $B$ to $\prod M_n/\oplus_{2, \omega} M_n$ (asymptotically) lift to discrete asymptotic homomorphisms to $\prod M_n$, then

1) all hyperlinear traces on $B$  are MF,

\medskip

2) if $A$ is  homotopy dominated by $B$, then all $\ast$-homomorphisms from $A$ to $\prod M_n/\oplus_{2, \omega} M_n$ (asymptotically) lift to discrete asymptotic homomorphisms to $\prod M_n$. In particular all hyperlinear traces on $A$  are MF.
\end{theorem}
\begin{proof} 1) Let $\tau$ be a hyperlinear trace on $B$. Then $\tau = tr \; f$, for some $\ast$-homomorphism $f: B \to \prod M_n/\oplus_{2, \omega} M_n$. By assumption $f$ lifts (asymptotically) to a discrete asymptotic homomorphism. By Proposition \ref{propositionMFtrace} $\tau$ is an MF trace.

2) Follows from Corollary \ref{HomInvGen} and 1).
\end{proof}

\begin{theorem}\label{HomLiftTraces2} If all $\ast$-homomorphisms from $B$ to $\prod M_n/\oplus_{2, \omega} M_n$ that have a cp lift to $\prod M_n$ (asymptotically) lift to cp discrete asymptotic homomorphisms to $\prod M_n$, then

1) all amenable  traces on $B$  are quasidiagonal,

\medskip

2) if $A$ is  homotopy dominated by $B$, then all $\ast$-homomorphisms from $A$ to $\prod M_n/\oplus_{2, \omega} M_n$ that have a cp lift to $\prod M_n$ (asymptotically) lift to cp discrete asymptotic homomorphisms to $\prod M_n$. In particular all amenable traces on $A$  are quasidiagonal.
\end{theorem}
\begin{proof} Same as proof of Theorem \ref{HomLiftTraces}, using Theorem \ref{cpHomotopyLiftingCor2} instead of Corollary \ref{HomInvGen}.
\end{proof}

\medskip

We will obtain some corollaries of Theorems \ref{HomLiftTraces} and \ref{HomLiftTraces}. For that we will need to show that the condition imposed on $B$ in these theorems reformulates in terms of traces when $B$ is exact.

\begin{proposition}\label{ExactLiftEquivTraces} Suppose $B$ is exact. TFAE:

(i) Every amenable trace on $B$ is quasidiagonal;

(ii) Every $\ast$-homomorphism from $B$ to $\prod M_n/\oplus_{2, \omega} M_n$ that has  a cp lift to $\prod M_n$, lifts to  a cp asymptotic homomorphism to $\prod M_n$.
\end{proposition}
\begin{proof} (i) $\Rightarrow$ (ii): Let $f: B \to \prod M_n/\oplus_{2, \omega} M_n$ be a $\ast$-homomorphism that lifts to a cp map to $\prod M_n$. Then $tr \;f$ is amenable. Hence it is quasidiagonal.  By Proposition \ref{propositionMFtrace} $tr \;f = tr \;g$, for some $\ast$-homomorphism $g:  B \to \prod M_n/\oplus_{2, \omega} M_n$ that lifts to a cp asymptotic homomorphism. By Corollary \ref{exact}, $f= u^*gu$, for some unitary $u\in \prod M_n/\oplus_{2, \omega} M_n$. Since $g$ lifts to a cp asymptotic homomorphism and $u$ lifts to a unitary in $\prod M_n$, $f$ also lifts to a cp asymptotic homomorphism.

(ii) $\Rightarrow$ (i): This is already proved in Theorem \ref{HomLiftTraces2}, without the exactness assumption.
\end{proof}

\begin{corollary}\label{BExact} If a C*-algebra $A$ is homotopically dominated by an exact C*-algebra $B$ and all amenable traces on $B$ are quasidiagonal, then all amenable traces on $A$ are quasidiagonal.
\end{corollary}
\begin{proof} By Proposition \ref{ExactLiftEquivTraces}, all $\ast$-homomorphisms from $B$ to $ \prod M_n/\oplus_{2, \omega} M_n$ that have a cp lift, lift to cp asymptotic homomorphisms. The statement follows now from Theorem \ref{HomLiftTraces2}.
\end{proof}

\begin{corollary}\label{BTWW} If a C*-algebra $A$ is homotopically dominated by an exact C*-algebra $B$ that has a faithful trace and satisfies UCT, then all amenable traces on $A$ are quasidiagonal.
\end{corollary}
\begin{proof} By Tikuisis-White-Winter Theorem (\cite{TWW}, \cite{Gabe}, \cite{Schafhauser}) all amenable traces on $B$ are quasidiagonal. The statement follows now from Corollary \ref{BExact}.
\end{proof}

\begin{corollary}\label{BNuclear} If a C*-algebra $A$ is homotopy dominated by a nuclear C*-algebra $B$ and all (hyperlinear) traces on $B$ are MF, then all hyperlinear traces on $A$ are MF.
\end{corollary}
\begin{proof} Since $B$ is nuclear, any trace on it is amenable, and any MF trace on it is quasidiagonal.  Thus all amenable traces on $B$ are quasidiagonal. Since nuclear C*-algebras are exact, by Proposition \ref{ExactLiftEquivTraces} any $\ast$-homomorphism from $B$ to $ \prod M_n/\oplus_{2, \omega} M_n$ that has a cp lift, lifts to a cp asymptotic homomorphism. Since $B$ is nuclear, any $\ast$-homomorphism has  a cp lift. Therefore any $\ast$-homomorphism from $B$ to $ \prod M_n/\oplus_{2, \omega} M_n$ lifts to an asymptotic homomorphism. By Theorem \ref{HomLiftTraces}, any hyperlinear trace on $A$ is MF.
\end{proof}

In fact, the condition imposed on $B$ in Theorems \ref{HomLiftTraces} and \ref{HomLiftTraces2} is satisfied for many C*-algebras. Recall that a C*-algebra $B$ is called {\it Hilbert-Schmidt stable}
if all $\ast$-homomorphisms from $B$ to $\prod M_n/\oplus_{2, \omega} M_n$ lift to $\ast$-homomorphisms to $\prod M_n$. By now lots of C*-algebras are known to be Hilbert-Schmidt stable, e.g. all type I C*-algebras, the C*-algebras of nilpotent groups and many other full group C*-algebras.
Since Hilbert-Schmidt stable C*-algebras satisfy the assumptions of Theorems \ref{HomLiftTraces} and \ref{HomLiftTraces2}, we obtain the following corollary.

\begin{corollary}\label{HilbertSchmidtStable} If $B$ is Hilbert-Schmidt stable (e.g. free product of type I C*-algebras) and $A$ is homotopy dominated by $B$, then all hyperlinear traces on $A$ are MF, and all amenable traces on $A$ are quasidiagonal.
\end{corollary}

%We note that Corollaries \ref{BExact}, \ref{BTWW},   \ref{HilbertSchmidtStable} all cover Brown-Carrion-White result (Proposition \ref{QDtraceCone})  and Corollaries \ref{BNuclear} and \ref{HilbertSchmidtStable} cover  Proposition \ref{MFtraceCone}, as particular case of $B=0$.

\medskip

\section{An application: Quasidiagonality}

\subsection{A characterization of quasidiagonality}
The following theorem gives a lifting characterization of quasidiagonality.

\begin{theorem}\label{characterizationQD} TFAE:

\begin{itemize}

\item[(i)] $A$ is QD,

\item[(ii)] every $\ast$-homomorphism from $A$ to $B(H)$ lifts to a cp discrete asymptotic homomorphism from $A$ to $\mathcal D$,

\item[(iii)] every $\ast$-homomorphism from $A$ to $B(H)$ asymptotically lifts to a cp discrete asymptotic homomorphism from $A$ to $\mathcal D$.

\item[(iv)] there exists an embedding of $A$ into $B(H)$ that asymptotically lifts to a cp discrete asymptotic homomorphism from $A$ to $\mathcal D$.
\end{itemize}
\end{theorem}
\begin{proof} Similar to the proof of Theorem \ref{characterizationMF}.

 (i)$\Rightarrow$ (ii):  Since $A$ is quasidiagonal, there exist cpc maps $\gamma_n: A \to M_n$ which are approximately multiplicative and approximately injective. Let $\pi: \prod M_n \to \prod M_n/\oplus M_n$ be the canonical surjection. Then $$\pi\circ \left(\gamma_n\right)_{n\in \mathbb N}: A \subset \prod M_n/\oplus M_n$$ is an embedding. Let $A_1 = \pi^{-1}(A)$. Let $pr_{ m}: \prod M_n \to \prod_{n\ge m} M_n$ be the projection map and
 $A_m = pr_{m}(A_1)$.   It is well-known and not hard to prove that $A    =\varinjlim A_n$ with $\theta_{1, m} = pr_m|_{A_1}$, $m\in \mathbb N$,  as connecting maps. Then   $s: = \left(\gamma_n\right)_{n\in \mathbb N}: A \to A_1$ is a cp lift of $\theta_{1, \infty}$.

Let $f: A \to B(H)$ be a $\ast$-homomorphism. By Theorem \ref{RFD} $f\circ \theta_{n, \infty}$ lifts to a $\ast$-homomorphism $\psi_n: A_n\to \mathcal D$.  For each $n\in \mathbb N$, define a cp map $\phi_n: A \to \mathcal D$ by
$$\phi_n(a) = \psi_n(\theta_{1, n}(s(a))),$$ $a\in A$. The rest of the proof is the same as in Theorem \ref{characterizationMF}, (i) $\Rightarrow$ (ii).

\medskip

(ii)$\Rightarrow$(iii) and  (iii)$\Rightarrow$(iv) are obvious.

\medskip

(iv)$\Rightarrow$(i): Same as the proof of Theorem \ref{characterizationMF}, (iv) $\Rightarrow$ (i).
\end{proof}

\begin{corollary}\label{QDhominv} (Voiculescu \cite{Voiculescu}) Suppose $A$ is homotopically dominated by $B$, and $B$ is quasidiagonal. Then $A$ is also quasidiagonal.
\end{corollary}
\begin{proof} Every $\ast$-homomorphism, say $f$,  from any C*-algebra to $B(H)$ has a cp lift to $\mathcal D$, namely $\left(P_nfP_n\right)_{n\in \mathbb N}$. The statement follows now from Theorem \ref{cpHomotopyLiftingCor2} and the equivalence (i)$ \Leftrightarrow$ (ii) in Theorem \ref{characterizationQD}.
\end{proof}

\subsection{$qA$ is quasidiagonal}

For a C*-algebra $A$ let $i_1: A \to A\ast A$ and $i_2: A \to A\ast A$ be the canonical embeddings.
In \cite{Cuntz} Cuntz defined $qA$ as the ideal in $A\ast A$ generated by the set $\{i_1(a) - i_2(a)\;|\; a\in A\}$. He proved that for any $B$,  the set $[qA, B\otimes K]$ of homotopy classes of $\ast$-homomorphisms from $qA$ to $B\otimes K$ is always a group with multiplication given by the direct sum,  and  $KK(A, B) \cong [qA, B\otimes K]$.

\noindent  We will see now that $qA$ is always quasidiagonal, regardless of $A$.

\begin{theorem}\label{qA} For any C*-algebra $A$, the C*-algebra $qA$ is quasidiagonal.
\end{theorem}
\begin{proof} Since quasidiagonality is a local property, it is sufficient to prove the statement for a separable $A$. Let $\rho: qA\to B(H)$ be a faithful representation. By the proof of \cite[Th. 5.1.6]{JensenThomsen} there exists a representation $\lambda: qA\to B(H)$ such that $\rho\oplus \lambda$ is homotopic to zero.

The identity map $id_{B(H)}$ admits a cp lift to $\mathcal D$, namely the map that sends each operator $T$ to the sequence $(P_nTP_n)_{n\in \mathbb N}$ of its corners. Therefore we can apply Theorem \ref{cpHomotopyLiftingCor1} to $B= B(H)$, $\phi= id_{B(H)}$, $D/I= B(H)$, $D=\mathcal D$, $\psi' = \rho\oplus \lambda$, $\psi'' = 0.$ By Theorem \ref{cpHomotopyLiftingCor1} $\rho\oplus \lambda$ lifts to a discrete asymptotic homomorphism $qA\to M_2(\mathcal D)$. By Theorem \ref{characterizationQD}, namely by the implication (iv) $\Rightarrow$ (i) applied to $M_2(\mathcal D)$ instead of $\mathcal D$, $qA$ is quasidiagonal.
\end{proof}

\medskip

\section{Liftings of asymptotic homomorphisms}

 In general we use the parameter $\lambda$ for asymptotic homomorphisms to  emphasize that they can be both discrete and continuous. In this section we use it  to distinguish the parameter from the variable $t$ used for functions on $[0, 1]$. We note that in Theorem \ref{DiscreteHomotopyLiftingAsHom} below the parameter has to be continuous.

\medskip

\medskip

\begin{lemma}\label{factorizationAsHom} Let $\phi_{\lambda}: B \to A$, $\lambda\in \Lambda$, be an asymptotic homomorphism and $\psi: B \to A$ be a $\ast$-homomorphism that is homotopic to $\phi_{\lambda}, \lambda\in \Lambda$. Then $\phi_{\lambda}$ factorizes through $Z_{\psi}$,

$$ \begin{tikzcd}  B \arrow{r}{\alpha_{\lambda}}\arrow[bend right=30,  swap]{rr}{\phi_{\lambda}}& Z_{\psi}\arrow{r}{\beta}& A \end{tikzcd}$$
meaning that there exist  an asymptotic homomorphism $\alpha_{\lambda}: B \to Z_{\psi}$, $\lambda\in \Lambda$,  and  a $\ast$-homomorphism $\beta: Z_{\psi}\to A$ such that $\phi_{\lambda} = \beta\circ \alpha_{\lambda}$.
Moreover $\beta|_B = \psi$.
\end{lemma}
\begin{proof} Similar to the proof of Lemma \ref{factorization}. Let $\Phi_{\lambda}: B\to A\otimes C[0, 1]$ be a homotopy between $\phi_{\lambda}$ and $\psi$ with $$ev_1\circ\Phi_{\lambda} = \phi_{\lambda}, \;ev_0\circ\Phi_{\lambda}=\psi.$$ Then we can define $\alpha_{\lambda}: B \to Z_{\psi}$ by
$$\alpha_{\lambda}(b) = (\Phi_{\lambda}(b), b),$$ for any $b\in B$.
Define $\beta: Z_{\psi}\to A$ by
$$\beta((\xi, b)) = \xi(1),$$ for any $(\xi, b)\in Z_{\psi}$. Then $\beta\circ \alpha_{\lambda} = \phi_{\lambda}$, $\lambda\in \Lambda$.
As usual we consider $B$ as a C*-subalgebra of $Z_{\psi}$ via the embedding $b\mapsto (1\otimes \psi(b), b)$. We have
$$\beta|_B(b) = \beta((1\otimes \psi(b), b))= \psi(b).$$
\end{proof}

\begin{theorem}\label{DiscreteHomotopyLiftingAsHom} Let $\phi_{\lambda}: B \to A$, $\lambda\in [1, \infty)$, be an asymptotic homomorphism and $\psi: B \to A$ be a $\ast$-homomorphism that is homotopic to $\phi_{\lambda}, \lambda\in [1, \infty)$. Suppose $\psi$ lifts  to an asymptotic homomorphism. Then $\phi_{\lambda}$, $\lambda\in [1, \infty)$, lifts  to an asymptotic homomorphism.
 Moreover the whole homotopy lifts.
\end{theorem}
\begin{proof} Similar to the proof of Theorem \ref{DiscreteHomotopyLifting}. To make the idea easier to understand, we at first prove the first statement and then the second one, although the first one follows from the second one.

We use the notation and constructions of Lemma \ref{factorizationAsHom}. By Lemma \ref{factorizationAsHom} $\phi_{\lambda} = \beta\circ \alpha_{\lambda}$, $\lambda\in \Lambda$,  with $\beta: Z_{\psi}\to D/I$ such that $\beta|_B = \psi$. Since $\psi$ lifts  to an asymptotic homomorphism, by Theorem \ref{cylinder}   $\beta$ lifts to an asymptotic homomorphism $\gamma_{\lambda}$, $\lambda\in \Lambda$. The composition $\gamma_{\lambda}\circ \alpha_{\lambda}$ might not be an asymptotic homomorphism, but
 there is a reparametrization  $r(\lambda)$ such that $\gamma_{r(\lambda)}\circ \alpha_{\lambda}$ is an asymptotic homomorphism. This is a lift of $\phi_{\lambda}$, $\lambda\in \Lambda$, since
  $$q\circ \gamma_{r(\lambda)}\circ \alpha_{\lambda} = \beta\circ \alpha_{\lambda} = \phi_{\lambda}. $$

We now show that the whole homotopy $\Phi_{\lambda}$, $\lambda\in \Lambda$, between $\phi_{\lambda}$, $\lambda\in \Lambda$, and $\psi$ lifts.  For each $\lambda\in \Lambda$, $0\le s\le 1$ and $b\in B$ we define $\Gamma_{\lambda, b, s}\in A\otimes C[0,1]$ by $$\Gamma_{\lambda, b, s}(t) = \Phi_{\lambda}(b)(st),$$ $t\in [0, 1]$. Since $\Gamma_{\lambda, b, s}(0) = \Phi_{\lambda}(b)(0) = \psi(b)$, $(\Gamma_{\lambda, b, s}, b) \in Z_{\psi}$. Since for each $\lambda\in \Lambda$ and $ b\in B$, $
 \Phi_{\lambda}(b)$ is continuous $A$-valued function on $[0, 1]$, the assignment $s\mapsto \Gamma_{\lambda, b, s}$ is continuous. Therefore we can define a map $\Theta_{\lambda}: B \to Z_{\psi}\otimes C[0, 1]$ by
$$\Theta_{\lambda}(b)(s) = (\Gamma_{\lambda, b, s}, b).$$ It is an asymptotic homomorphism. Indeed, for $b_1, b_2 \in B$ we have
\begin{multline*}\|\Theta_{\lambda}(b_1)\Theta_{\lambda}(b_2)\| =   \sup_{s\in [0, 1]} \|\left(\Gamma_{\lambda, b_1, s}, b_1\right)\left(\Gamma_{\lambda, b_2, s}, b_2\right) - \left(\Gamma_{\lambda, b_1b_2, s}, b_1b_2\right)\| \\=
\|b_1b_2\|\sup_{s\in [0, 1]}\|\Gamma_{\lambda, b_1, s}\Gamma_{\lambda, b_2, s}-\Gamma_{\lambda, b_1b_2, s}\| \\= \|b_1b_2\|\sup_{s\in [0, 1]}
\sup_{t\in [0, 1]} \|\Phi_{\lambda}(b_1)(st)\Phi_{\lambda}(b_2)(st)-\Phi_{\lambda}(b_1b_2)(st)\| \\
\le \|b_1b_2\| \|\Phi_{\lambda}(b_1)\Phi_{\lambda}(b_2)-\Phi_{\lambda}(b_1b_2)\| \to_{\lambda \to \infty} 0, \end{multline*}
asymptotic linearity and asymptotic self-adjointness can be checked similarly, and for each $\lambda_1, \lambda_2 \in \Lambda$ and $b\in B$ we have
\begin{multline*}\|\Theta_{\lambda_1}(b) - \Theta_{\lambda_2}(b)\| =\|b\| \sup_{s\in [0, 1]} \|\Gamma_{\lambda_1, b, s} - \Gamma_{\lambda_2, b, s}\|\\ = \|b\| \sup_{s\in [0, 1]} \sup_{t\in [0, 1]} \|\Phi_{\lambda_1}(b)(st) - \Phi_{\lambda_2}(b)(st)\| \le \|b\|\|\Phi_{\lambda_1}(b) - \Phi_{\lambda_2}(b)\| \end{multline*}
that implies that the function $\lambda \mapsto \Theta_{\lambda}(b)$ is continuous.

We have  $\Phi_\lambda = (\beta\otimes id_{C[0, 1]})\circ \Theta_{\lambda},$ $\lambda\in \Lambda$.  There exists a reparametrization $r(\lambda)$ such that $(\gamma_{r(\lambda)}\otimes id_{C[0, 1]})\circ \Theta_{\lambda}: B\to D\otimes C[0, 1]$, $\lambda\in \Lambda$,  is an asymptotic homomorphism.   This is a homotopy that lifts $\Phi_{\lambda}$.
\end{proof}

\begin{remark} Same result holds if we replace lifting by asymptotic lifting.
\end{remark}

\begin{remark}\label{noDiscreteVersion} We do not have a version of Theorem \ref{DiscreteHomotopyLiftingAsHom} for discrete asymptotic homomorphisms. The same proof would not work in the discrete case, because the composition of discrete homomorphisms need  not be an asymptotic homomorphism even after reparametrization of one of them.
\end{remark}

\medskip

%It is interesting to compare our homotopy lifting theorem with Carrion-Schafhauser homotopy lifting theorem for asymptotic homomorphism \cite{CSch}. They proved that if $A$ is inductive limit of SP, $\phi_t, \psi_t: A \to B/I$ are homotopic asymptotic homomorphisms, and $\phi_t$ lifts asymptotically, then $\psi_t$ lifts asymptotically. We have one asymptotic homomorphism and one actual homomorphism but instead we have arbitrary $A$ and lifting instead of asymptotic lifting. For example our theorem implies that any asymptotic homomorphism from a cone lifts, while their theorem says that it only lifts asymptotically.

%\medskip

\section{An application: Homotopy symmetric C*-algebras are MF}

Homotopy symmetric C*-algebras were introduced by Dadarlat and Loring in \cite{DadarlatLoring}. Let $[[A, B]]$  denote the set of homotopy classes of asymptotic homomorphisms from $A$ to $B$. When $B$ is stable, i.e. $B\cong B\otimes K$, this set is a semigroup with addition given by the direct sum.  A C*-algebra is {\it homotopy symmetric} if $id_A$ has an inverse in $[[A, A\otimes K]]$. (More precisely, the class of the map $a\mapsto a\otimes e_{11}$ has an inverse).

A remarkable result of Dadarlat and Loring states that homotopy symmetric C*-algebras are precisely those  C*-algebras for which E-theory can be unsuspended: for any $B$, $E(A, B) = [[A, B\otimes K]]$ if and only if $A$ is homotopy symmetric. Recall that E-theory was introduced by Connes and Higson in \cite{ConnesHigson}  and is defined as $E(A, B) = [[SA, SB\otimes K]]$.

In \cite{DadarlatPennig} Dadarlat and Pennig proved that nuclear homotopy symmetric C*-algebras are quasidiagonal.  We prove now that every homotopy symmetric C*-algebra is MF. Since for nuclear C*-algebras the MF property is equivalent to quasidiagonality, this covers the aforementioned result of Dadarlat and Pennig.

Let us say that an asymptotic homomorphism $\phi_t, t\in [0, \infty),$ is {\it injective} if $\limsup \|\phi_t(a)\|\ge \|a\|$, for any $a\in A$.

Let $\mathcal D$ be the same  C*-algebra as in Theorem \ref{characterizationMF}.

By arguments similar to those used in the proof of the implication (iv)$\Longrightarrow$ (i) in Theorem \ref{characterizationMF},  we obtain the following proposition.

\begin{proposition}\label{extraMF} Let $A$ be a separable C*-algebra and suppose there exists an injective asymptotic homomorphism from $A$ to $B(H)$ that asymptotically lifts to an asymptotic homomorphism to $\mathcal D$. Then $A$ is MF.
\end{proposition}
\begin{proof}
Let $F\subset \subset A$, $\epsilon >0$. Let $\phi_t: A \to B(H)$, $t\in [0, \infty)$,  be an injective asymptotic homomorphism that asymptotically lifts to an asymptotic homomorphism $\psi_t: A \to \mathcal D$, $t\in [0, \infty)$. Then there is $t_0$ such that for any $t>t_0$
\begin{equation}\label{characterizationMF1}\|\psi_t(ab) - \psi_t(a)\psi_t(b)\|\le \epsilon,\end{equation} when $a, b\in F$, and similarly for sums and adjoint element. Since for each $a\in A$, $\|q\circ\psi_t(a)- \phi_t(a)\|\to 0$,  there is $t_1> t_0$ such that for any $t>t_1$
\begin{equation} \|q\circ \psi_t(a)\|\ge \|\phi_t(a)\|-\epsilon.\end{equation}
By the injectivity of $\phi_t$, $t\in [0, \infty)$, there is $t_2> t_1$ such that
\begin{equation} \|\phi_{t_2}(a)\|\ge \|a\|-\epsilon, \end{equation}
for any $a\in F$.  Then
\begin{equation}\label{norm} \|\psi_{t_2}(a)\|\ge \|q\circ\psi_{t_2}(a)\|\ge \|\phi_{t_2}(a)\|-\epsilon \ge \|a\|-2\epsilon,
\end{equation}
for any $a\in F$.  Write $\psi_{t_2} = \left(\psi_{t_2, k}\right)_{k\in \mathbb N},$ where $\psi_{t_2, k}: A \to M_k$. It follows from (\ref{norm}) that there is $k$ such that
\begin{equation}\label{characterizationMF3}\|\psi_{t_2, k}(a)\|>\|a\|-3\epsilon,\end{equation} $a\in F$. It follows from (\ref{characterizationMF1}) that
\begin{equation}\label{characterizationMF4}\|\psi_{t_2,k}(ab) - \psi_{t_2, k}(a)\psi_{t_2, k}(b)\|\le \epsilon,\end{equation}
$a, b \in F$.
By (\ref{characterizationMF3}) and (\ref{characterizationMF4}), $A$ is MF.
\end{proof}

\begin{theorem}\label{homsymm} Let $A$ be a homotopy symmetric C*-algebra. Then $A$ is MF.
\end{theorem}
\begin{proof} Since $A$ is homotopy symmetric, there is an asymptotic homomorphism $\gamma_t$, $t\in [0, \infty)$, such that $id_A \oplus \gamma_t$ is homotopic to zero. Embedding $A\otimes K\subset B(H)$ and considering $\gamma_t$, $t\in [0, \infty)$, as an asymptotic homomorphism to $B(H)$, we obtain an injective asymptotic homomorphism $\phi_t:= id_A\oplus \gamma_t: A \to B(H)$, $t\in [0, \infty)$, that is homotopic to zero.

Let $A_0$ be a separable C*-subalgebra of $A$. Then $\phi_t|_{A_0}$, $t\in [0, \infty)$, is an injective asymptotic homomorphism that is homotopic to zero. Since zero homomorphism lifts to the zero homomorphism from $A_0$ to $D$, by Theorem \ref{DiscreteHomotopyLiftingAsHom}, $\phi_t|_{A_0}$, $t\in [0, \infty)$, also lifts to an asymptotic homomorphism to $\mathcal D$. By Proposition  \ref{extraMF}, $A_0$ is MF. We have proved that any separable C*-subalgebra of $A$ is MF. Since MF is a local property, this implies that $A$ is MF.
\end{proof}

 \section{An application: Extension groups}

 In this section a C*-algebra $B$ is always assumed to be stable, that is $B\cong B\otimes K$.
 We will use notation $Q(B)$ for $M(B)/B$.

 An extension of $A$ by $B$ can be described by its Busby invariant, a homomorphism from $A$ to $Q(B)$.

 The following information is taken from  the paper \cite{Manuilov} that contains an excellent introduction to the topic of classification of extensions.

 For classification of extensions, one must decide  what extensions should be considered as trivial, and
what should equivalence relation mean. There are various choices for both questions.  An extension $\lambda: A \to Q(B)$ should be considered as “trivial” if it is

 \medskip

(0) a zero extension, i.e., if $\lambda = 0$;

(1) a split extension, i.e., if $\lambda$ admits a $\ast$-homomorphism as a lifting. This means that there exists
a $\ast$-homomorphism $f: A \to M(B)$ such that $q\circ f = \lambda$;

(2) an asymptotically split extension, i.e., if $\lambda$ admits an asymptotic homomorphism as a lifting. This
means that there exists an asymptotic homomorphism $(f_t)_{t\in [0, \infty)}: A \to M(B)$ such that $q\circ  f_t = \lambda$
for every $t$;

(3) a discretely asymptotically split extension, i.e., if $\lambda$ admits a discrete asymptotic homomorphism
$(f_n)_{n\in \mathbb N}$ as a lifting.
This is almost the same as the previous one, but the parameter is integer and
there is no relation between $f_n$ and $f_{n+1}$.

\medskip

There are also several notions of equivalence for extensions. Let $\tau_0, \tau_1: A \to Q(B)$.

\medskip

(a) Unitary equivalence: $\tau_0 \approx \tau_1$ if there exists a unitary element $U\in M(B)$ such that $q(U)^* \tau_0q(U) = \tau_1$.
%This equivalence is most natural, although the space of classes of unitary equivalence is not a Hausdorff space.

(b) Approximate unitary equivalence: $\tau_0 \approx \tau_1$ if there exists a sequence $U_n\in M(B)$ of unitary elements
such that for any $a\in A$, one has $$\lim_{n\to \infty}  \|q(U_n)^* \tau_0(a) q(U_n) - \tau_1(a)\| = 0$$  (another version of approximate
unitary equivalence requires continuous families of unitary elements instead of sequences: $$\lim_{t\to \infty}  \|q(U_t)^* \tau_0(a) q(U_t) - \tau_1(a)\| = 0,$$ where $U_t\in M(B)$, $t\in [1, \infty)$,  is a continuous path of unitaries).

(c) Homotopy equivalence: $\tau_0 \approx \tau_1$ if there exists a $\ast$-homomorphism $T : A \to Q(B)[0, 1]$ such that
$ev_i \circ T = \tau_i$ for $i = 0, 1$.

(d) Weak homotopy equivalence: $\tau_0 \approx \tau_1$ if there exists a $\ast$-homomorphism $T: A \to Q(B[0, 1])$ such that
$ev_i \circ T = \tau_i$ for $i = 0, 1$ (although the C*-algebra $Q(B[0, 1])$ is larger than $Q(B)[0, 1]$, the evaluation
mappings are still well defined).

\medskip

Two extensions $\tau_0, \tau_1: A \to Q(B)$ are called {\it stably equivalent} if there exist two "trivial" extensions $\lambda_0, \lambda_1: A \to Q(B)$ such that $\tau_0\oplus \lambda_) \approx \tau_1\oplus \lambda_1$.

 One denotes by $Ext(A, B)$ the set of classes of stably equivalent extensions of an algebra $A$ by an algebra $B$ and indicates by subscripts which version of equivalence and triviality is used. For example, $Ext_{1a}(A, B)$ means
that we consider split extensions as “trivial” extensions and the equivalence is the unitary equivalence. For the continuous version of the approximate unitary equivalence we will use the notation $Ext_{\ast b^{cont}}$. $Ext_{\ast\ast}(A, B)$ is always a semigroup (since $B$ is stable).

 \bigskip

 \medskip

  In \cite{MT15} Manuilov and Thomsen obtained  the following result.

 \begin{theorem}(\cite{MT15}, \cite[Th. 3.3]{Manuilov}) Let $A$ be a separable C*-algebra and $B$ be a $\sigma$-unital C*-algebra. Then all $Ext_{\ast\ast}(SA, B)$ (except for the cases where $\ast\ast$ is (0a), (0b), (1a) and (1b)) coincide and are groups. \end{theorem}

 Here we will show that some of those equalities  can be unsuspended.

 \begin{lemma}\label{unit} Asymptotically split extensions exist and represent the unit element of $Ext_{2, \ast}(A, B)$.
 \end{lemma}
 \begin{proof} Since $B\cong B\otimes K$, $Q(B) \supset B(H)/K(H)$. Let $\pi: A\to B(H)$ be a faithful representation. Then the extension
 $\lambda: = q\circ \pi^{(\infty)}$ splits (not only asymptotically).

 For any $\tau: A \to Q(B)$ we have $$(\tau\oplus \lambda)\oplus \lambda \approx \tau \oplus (\lambda\oplus \lambda),$$ with both $\lambda$ and $\lambda\oplus \lambda$ being asymptotically split. Therefore $[\tau]+[\lambda] = [\tau\oplus\lambda] = [\tau].$ Thus $[\lambda]$ is the unit of
 $Ext_{2, \ast}(A, B)$.
 \end{proof}

 \begin{lemma}\label{variousEquivalencesExtensiions} Let $\tau_0, \tau_1: A \to Q(B)$ be two extensions. Then
 \begin{itemize}
 \item[(i)] If $\tau_0$ and  $\tau_1$ are unitary equivalent, then they are homotopic,

  \item[(ii)] If $\tau_0$ and  $\tau_1$ are continuously approximately unitary equivalent, then they are homotopic,

  \item[(iii)] If $\tau_0$ and  $\tau_1$ are approximately unitary equivalent, then they are homotopic as discrete asymptotic homomorphisms.

  \end{itemize}
 \end{lemma}
 \begin{proof} (i): Suppose $q(U)^* \tau_0q(U) = \tau_1$. Since for a stable C*-algebra $B$, the unitary group of $M(B)$ is connected \cite{Mingo}, there is a continuous path of unitaries $U_t\in M(B)$ that connects $U$ with $\mathbb 1$. Then $T : A \to Q(B)[0, 1]$ defined by
 $$T(a)(t) = q(U_t)^* \tau_0(a) q(U_t)$$ is a homotopy between $\tau_1$ and $\tau_0$.

 (ii): Suppose $\lim_{t\to \infty}  \|q(U_t)^* \tau_0(a) q(U_t) - \tau_1(a)\| = 0$, where $U_t\in M(B)$, $t\in [1, \infty)$,  is a continuous path of unitaries. Let $\gamma: [1/2, 1]\to \mathcal U(M(B))$ be a continuous path such that $\gamma(1)= \mathbb 1, \gamma(1/2) = U_1$.
  For $a\in A$ we define an $M(B)$-valued function $\Phi(a)$ on $[0, 1]$ by

 $$\Phi(a)(t) = \begin{cases}  q(\gamma(t)^*)\tau_0(a) q(\gamma(t)), \;t \ge 1/2\\
 q(U_t)^*\tau_0(a)q(U_t), \;0<t<1/2\\
 \tau_1(a), \;t=0.
 \end{cases}$$
 Then $\Phi(a)\in M(B)[0, 1]$, for each $a\in A$. Define $\Phi: A \to M(B)[0, 1]$ by $a\mapsto \Phi(a)$. Then $\Phi$ is a $\ast$-homomorphism,  and $\Phi(a)(1) = \tau_0(a)$, $\Phi(a)(0) = \tau_1(a)$, for any $a\in A$.

 (iii): Suppose $\lim_{n\to \infty}  \|q(U_n)^* \tau_0(a) q(U_n) - \tau_1(a)\| = 0$, for some sequence of unitaries $U_n\in M(B)$. Let us denote $q(U_n)$ by $u_n$, for short. For $\lambda\in \mathbb N$ let $\gamma^{\lambda}: [1/\lambda, 1]\to \mathcal U(M(B))$ be a continuous path such that $\gamma^{\lambda}(1) = \mathbb 1, \gamma^{\lambda}(1/\lambda)= U_{1/\lambda}.$ For $\lambda\in \mathbb N$ and $a\in A$ we define an $M(B)$-valued function $\Phi_{\lambda}(a)$ on $[0, 1]$ by
 $$\Phi_{\lambda}(a)(t) = \begin{cases}  q(\gamma^{\lambda}(t))^*\tau_0(a)q(\gamma^{\lambda}(t)), \;t \ge 1/\lambda\\
 su_n^*\tau_0(a)u_n + (1-s)u_{n+1}^*\tau_0(a)u_{n+1}, \; 1/\lambda > t>0 \\ \;\text{and}\; t = s\frac{1}{n} + (1-s)\frac{1}{n+1}, \;\text{for some}\; s\in [0, 1]\\
 \tau_1(a), \;t=0.
 \end{cases}$$

 Clearly the function $\Phi_{\lambda}(a)$ is continuous at any point $t\neq 0$. Let us show that it is continuous also at $t=0$. Let $\epsilon>0$. There is $n_0$ such that for any $n> n_0$ one has
 $$\|\tau_1(a) - u_n^*\tau_0(a)u_n\|<\epsilon.$$ Then for $t< \min\{1/\lambda, 1/{n_0}\}$
 \begin{multline}\label{variousEquivalencesExtensiions1}  \|\Phi_{\lambda}(a)(t) - \Phi_{\lambda}(a)(0)\| = \|\Phi_{\lambda}(a)(t)-\tau_1(a)\|\\
 \|s(u_n^*\tau_0(a)u_n - \tau_1(a)) + (1-s)(u_{n+1}^*\tau_0(a)u_{n+1} - \tau_1(a))\|\le \epsilon.
 \end{multline}
 Thus $\Phi_{\lambda}(a)\in M(B)[0, 1]$. We define $\Phi_{\lambda}: A \to M(B)[0, 1]$ by $a\mapsto \Phi_{\lambda}(a)$.
 Let us show that $\Phi_{\lambda}$, $\lambda\in \mathbb N$, is a discrete asymptotic homomorphism.
 Fix $a_1, a_2\in A$ and $\epsilon>0$. Then for any $\lambda> n_0$, and for $ t = s\frac{1}{n}+(1-s)\frac{1}{n+1}< 1/{\lambda}$, we obtain, using (\ref{variousEquivalencesExtensiions1}),
 \begin{multline*}\|\left(\Phi_{\lambda}(a_1)\Phi_{\lambda}(a_2) - \Phi_{\lambda}(a_1a_2)\right)(t)\| \\= \|(\Phi_{\lambda}(a_1)(t)-\tau_1(a_1))\Phi_{\lambda}(a_2)(t) + \tau_1(a_1)(\Phi_{\lambda}(a_2)(t)-\tau_1(a_2)) + \tau_1(a_1a_2) - \Phi_{\lambda}(a_1a_2)(t)\| \\ \le \epsilon \max\{\|\Phi_{\lambda}(a_2)\|, \|a_1\|, 1\}.
 \end{multline*}
 If either $t> 1/{\lambda}$ or $t=0$, then
 $$\left(\Phi_{\lambda}(a_1)\Phi_{\lambda}(a_2) - \Phi_{\lambda}(a_1a_2)\right)(t) =0.$$
 This shows that $$\lim_{\lambda\to 0} \|\Phi_{\lambda}(a_1)\Phi_{\lambda}(a_2)-\Phi_{\lambda}(a_1a_2)\|=0, $$
 for any $a_1, a_2\in A$. Asymptotic linearity and self-adjointness are checked similarly.

 \end{proof}

The following lemma establishes the connection between various versions of $Ext_{\ast\ast}$ being a group.

\begin{lemma}\label{connectionGroups} \begin{itemize}

\item[(1)]  $Ext_{2\ast}(A, B)$ is a group $\Rightarrow$  $Ext_{3\ast}(A, B)$ is a group.

\item[(2)]    $Ext_{2b^{cont}}(A, B)$ is a group $\Rightarrow$  $Ext_{2b}(A, B)$ is  a group.

\item[(3)]  $Ext_{3b^{cont}}(A, B)$ is a group $\Rightarrow$  $Ext_{3b}(A, B)$ is  a group.

\item[(4)]  $Ext_{2a}(A, B)$ is a group $\Leftrightarrow$  $Ext_{2b^{cont}}(A, B)$ is a group $\Leftrightarrow$  $Ext_{2c}(A, B)$ is  a group.

\item[(5)]  $Ext_{3a}(A, B)$ is a group $\Leftrightarrow$  $Ext_{3b^{cont}}(A, B)$ is a group $\Leftrightarrow$  $Ext_{3c}(A, B)$ is  a group.

\end{itemize}
\end{lemma}
\begin{proof} 1), 2), 3) follow from the definition of $Ext_{\ast\ast}$.

\medskip

4) From Lemma \ref{variousEquivalencesExtensiions} and the definition of $Ext_{\ast\ast}$ it follows that $Ext_{2a}(A, B)$ is a group $\Rightarrow$  $Ext_{2b^{cont}}(A, B)$ is a group $\Rightarrow$  $Ext_{2c}(A, B)$ is  a group.
We need to prove the implication $Ext_{2c}(A, B)$ is a group $\Rightarrow$  $Ext_{2a}(A, B)$ is  a group.

Let $[\tau]_{2a}\in Ext_{2a}(A, B)$. Since $Ext_{2c}(A, B)$ is a group, there exists an extension $\tau'$ such that $[\tau]_{2c} + [\tau']_{2c}=e_{2c}$, where by $e_{2c}$ we denote the unit element of $Ext_{2c}(A, B)$. By Lemma \ref{unit}, this means that there exist asymptotically split extensions $\lambda$, $\lambda'$ and $\lambda''$ such that $\tau\oplus\tau'\oplus\lambda'$ is homotopic to $\lambda\oplus \lambda''$. Since $\lambda\oplus \lambda''$ asymptotically splits, by Theorem \ref{DiscreteHomotopyLifting} so does $\tau\oplus\tau'\oplus\lambda'$. This implies $[\tau]_{2a} + [\tau'\oplus\lambda']_{2a}=e_{2a}$ and thus $[\tau]_{2a}$ has an inverse element.

\medskip

5) is proved along the same lines as 4).

\end{proof}

 \begin{theorem}\label{ExtCoincide} Let $A$ be a separable C*-algebra and $B$   any C*-algebra.
 If any of $Ext_{2, a}(A, B),\; Ext_{2, b^{cont}}(A, B), \;Ext_{2, c}(A, B)$ is a  group, then they all coincide.
 %$$Ext_{2, a}(A, B) = Ext_{2, b}(A, B) = Ext_{2, c}(A, B).$$
 If any of $Ext_{3, a}(A, B),\; Ext_{3, b^{cont}}(A, B), \;Ext_{3, c}(A, B)$ is a group, then they all coincide.
 %$$Ext_{3, a}(A, B) = Ext_{3, b} (A, B)= Ext_{3, c}(A, B).$$
 \end{theorem}
 \begin{proof}
 There are natural well-defined surjective homomorphisms  $$j_1: Ext_{2, a}(A, B) \to  Ext_{2, b^{cont}}(A, B)$$ and,
  by Lemma \ref{variousEquivalencesExtensiions}, $$j_2: Ext_{2, b^{cont}}(A, B)\to Ext_{2, c}(A, B)$$ that send the class of an extension to the class of the same extension. We need to prove that they are injective. It is sufficient to prove that the homomorphism $j:= j_2\circ j_1: Ext_{2, a}(A, B) \to
 Ext_{2, c}(A, B)$  is injective. By the assumption and Lemma \ref{connectionGroups} 4), both  $Ext_{2, a}(A, B)$ and $Ext_{2, c}(A, B)$ are groups. Hence it is sufficient to check that the kernel  of $j$ consists of the unit element $e_{2,a}$ of $Ext_{2, a}(A, B)$. So suppose $[\tau]_{2,c} = j([\tau]_{2, a}) = e_{2, c}.$ By Lemma \ref{unit}, $e_{2, c} = [\lambda]_{2, c}$, for some (in fact, any) asymptotically split extension $ \lambda$. Therefore there are asymptotically split extensions $\lambda', \lambda''$ such that
 $$\tau \oplus \lambda' \sim_h \lambda\oplus \lambda''.$$ Since $\lambda\oplus \lambda''$ asymptotically splits,  by Theorem \ref{DiscreteHomotopyLifting} so does  $\tau \oplus \lambda'$. By Lemma \ref{unit}, $ e_{2, a}= [\tau \oplus \lambda'] = [\tau].$

 The second statement is proved along the same lines.
 \end{proof}

 \begin{remark} It does not seem possible to use the statement (iii) from Lemma \ref{variousEquivalencesExtensiions} to compare  $Ext_{3, b}(A, B)$ with  $Ext_{3, a}(A, B)$  and $Ext_{3, c}(A, B)$. The reason is that the homotopy lifting theorem for asymptotic homomorphisms (Theorem \ref{DiscreteHomotopyLiftingAsHom}) does not seem to admit a discrete version -- see Remark \ref{noDiscreteVersion}.
 %{\color{red} To be able to compose them, one needs the inner one to satisfy the additional assumption that  $\|\Phi_{m+1}(a) - \Phi_m(a)\| \to_{m\to \infty} 0$ and the as.hom. $\Phi_n$ constructed in Lemma \ref{variousEquivalencesExtensiions} does not satisfy this, unfortunately. }
 \end{remark}

 \begin{corollary}
 Let $A$ be a separable C*-algebra and $B$   any C*-algebra. Then
 \begin{itemize}

\item[(1)] $Ext_{2, a}(qA, B) = Ext_{2, b^{cont}}(qA, B) = Ext_{2, c}(qA, B)$,

\item[(2)]  $Ext_{3, a}(qA, B)= Ext_{3, b^{cont}}(qA, B)= Ext_{3, c}(qA, B)$.
\end{itemize}
 \end{corollary}
 \begin{proof} Since $Q(B\otimes K) \cong M_2(Q(B\otimes K))$, the proof of \cite[Th. 5.1.6]{JensenThomsen} shows that for any extension $\tau: qA \to Q(B\otimes K)$
there is an extension $\tau': qA \to Q(B\otimes K)$ such that $\tau\oplus \tau'$ is homotopic to zero. Therefore $[\tau]_{2c} + [\tau']_{2c} = e_{2c}$ and $[\tau]_{3c} + [\tau']_{3c} = e_{3c}$. Thus $Ext_{2c}(qA, B)$ and $Ext_{3c}(qA, B)$ are groups. The statement follows now from Theorem \ref{ExtCoincide}.

 \end{proof}

While $Ext_{\ast\ast}(A, B)$ need not be a group in general,  one can consider
its subset $Ext_{\ast\ast}^{-1}(A, B)$ of invertible elements, which is always a group. The following theorem is proved by the same argument as Theorem \ref{ExtCoincide}.

 \begin{theorem} Let $A$ be a separable C*-algebra and $B$   any C*-algebra. Then
 \begin{itemize}

\item[(1)] $Ext_{2, a}^{-1}(A, B) = Ext_{2, b^{cont}}^{-1}(A, B) = Ext_{2, c}^{-1}(A, B)$,

\item[(2)]  $Ext_{3, a}^{-1}(A, B)= Ext_{3, b^{cont}}^{-1}(A, B)= Ext_{3, c}^{-1}(A, B)$.
\end{itemize}
\end{theorem}

\subsection{Non-semi-invertible extensions}

To construct an extension that does not have an inverse in some of the semigroups $Ext_{\ast\ast}$ is a highly non-trivial task. Several examples of extensions that are not invertible in the classical extension semigroup $Ext_{1a}$
 are now known. One of the first such examples was constructed by Wassermann \cite{Wassermann} and involved property (T) groups.

In  \cite{MT2} Manuilov and Thomsen proved that Wassermann's extension is not even invertible in the semigroup $Ext_{2a}$. In other words, its direct sum with any  extension does not merely fail to split; it does not even asymptotically split.  They called such extensions non-semi-invertible. They also proved that a certain modification of Wassermann's extension is not invertible even up to homotopy, that is non-invertible in $Ext_{1c}$.

We show here that there is no need to consider a modification of Wassermann's extension -- Wassermann's extension itself is not invertible up to homotopy. In fact, it is not even semi-invertible up to homotopy, that is, it is non-invertible in $Ext_{2c}$, as the following proposition shows.

\begin{proposition} If an extension is not semi-invertible, then it is also not semi-invertible up to homotopy.
\end{proposition}
\begin{proof}
Suppose an extension $\tau: A \to Q(B)$ is semi-invertible up to homotopy. This means that there is an extension $\tau': A \to Q(B)$ such that $\tau\oplus \tau'$ is homotopic to an extension that asymptotically splits. By Theorem \ref{DiscreteHomotopyLifting}, $\tau\oplus \tau'$ itself asymptotically splits. Hence $\tau$ is semi-invertible.
\end{proof}

\end{document}